\documentclass[a4paper,12pt]{article}
\usepackage{amssymb,amsmath,mathrsfs,euscript,amsopn}

\numberwithin{equation}{section}

\newtheorem{lemma}{Lemma}[section]

\newtheorem{theorem}{Theorem}[section]
\newtheorem{proposition}{Proposition}[section]
\newtheorem{corollary}{Corollary}[section]
\newtheorem{definition}{Definition}[section]
\newtheorem{remark}{Remark}[section]

\renewcommand{\d}{\mathrm d}
\newcommand{\e}{\mathrm e}
\renewcommand{\i}{\mathrm i}
\renewcommand{\o}{\mathrm o}
\newcommand{\Or}{\mathrm O}

\renewcommand{\Im}{\operatorname{Im}}
\newcommand{\supp}{\operatorname{supp}}
\newcommand{\res}{\operatorname{res}}
\newcommand{\rank}{\operatorname{rank}}
\newcommand{\lin}{\operatorname{lin}}
\newcommand{\bq}{\mathbf q}
\newcommand{\Ran}{\operatorname{Ran}}
\newcommand{\slim}{\operatornamewithlimits{s-lim}}

\newenvironment{proof}{\smallskip\noindent{\bf Proof.}\rm}{\hspace*{\fill} $\Box$\medskip}
\newenvironment{proofsketch}{\smallskip\noindent{\it Sketch of the proof.}\rm}{\hspace*{\fill} $\Box$\medskip}

\newenvironment{proofTh1}{\smallskip\noindent{\it Proof of Theorem~\ref{Th1}.}\rm}{\hspace*{\fill} $\Box$\medskip}
\newenvironment{proofTh2}{\smallskip\noindent{\it Proof of Theorem~\ref{Th2}.}\rm}{\hspace*{\fill} $\Box$\medskip}
\newenvironment{proofTh3}{\smallskip\noindent{\it Proof of Theorem~\ref{Th3}.}\rm}{\hspace*{\fill} $\Box$\medskip}
\newenvironment{proofTh4}{\smallskip\noindent{\it Proof of Theorem~\ref{Th4}.}\rm}{\hspace*{\fill} $\Box$\medskip}
\newenvironment{proofB1}{\smallskip\noindent{\it Proof of Proposition~\ref{B1Lemma}.}\rm}{\hspace*{\fill} $\Box$\medskip}
\newenvironment{proofB2}{\smallskip\noindent{\it Proof of Proposition~\ref{B2Lemma}.}\rm}{\hspace*{\fill} $\Box$\medskip}
\newenvironment{proofB3}{\smallskip\noindent{\it Proof of Theorem~\ref{B3Th}.}\rm}{\hspace*{\fill} $\Box$\medskip}

\title{Inverse spectral problems for Dirac operators\\ with summable matrix-valued potentials}

\author{
D.~V.~Puyda\thanks{\emph{Email addresses:} dpuyda@gmail.com (D.~V.~Puyda)}\\
\emph{Ivan Franko National University of Lviv}\\
\emph{1 Universytetska str., Lviv, 79000, Ukraine}
}

\date{}

\begin{document}

\maketitle

\begin{abstract}
We solve the direct and inverse spectral problems for Dirac operators on $(0,1)$ with matrix-valued potentials whose entries belong to $L_p(0,1)$, $p\in[1,\infty)$. We give a complete description of the spectral data (eigenvalues and suitably introduced norming matrices) for the operators under consideration and suggest a method for reconstructing the operator from the spectral data.
\end{abstract}

\section{Introduction}

The aim of the present paper is to solve the direct and inverse spectral problems for self-adjoint Dirac operators generated by the differential expressions
$$
\mathfrak t_q:=
\frac{1}{\i}\begin{pmatrix}I & 0\\0 & -I\end{pmatrix}\frac{\d}{\d x}+
\begin{pmatrix}0 & q\\q^* & 0\end{pmatrix}
$$
and some separated boundary conditions. Here $q$ is an $r\times r$ matrix-valued function with entries belonging to $L_p(0,1)$, $p\in[1,\infty)$, called the potential of the operator, and $I$ is the identity $r\times r$ matrix. For such Dirac operators, we introduce the notion of the spectral data -- eigenvalues and specially defined norming matrices. We then give a complete description of the class of the spectral data for the operators under consideration, show that the spectral data determine the operator uniquely and provide an efficient method for reconstructing the operator from the spectral data.

Direct and inverse spectral problems for Sturm--Liouville and Dirac operators \cite{LevSargs1,Marchenko,Thaller} have been studied over the last 50 years. Already in 1966, M.~Gasymov and B.~Levitan suggested using the spectral function or the scattering phase for reconstructing Dirac operators on a half-line \cite{GasLev1,GasLev2}. Ever since, direct and inverse spectral problems for Dirac operators, including the systems of higher orders, have been considered in many papers.
Among the recent investigations in that area we mention the ones by M.~Malamud et al. \cite{Malamud1,Malamud2,Malamud3}, F.~Gesztesy et al. \cite{ClarkGesz,KisMakGesz,GeszOpV}, A.~Sakhnovich \cite{Sakh1,Sakh2}. Problems similar to the ones considered in this paper were recently treated by D.~Chelkak and E.~Korotyaev for Sturm-Liouville operators with matrix-valued potentials \cite{Korot1,Korot2}. We refer the reader to the references in \cite{Malamud1}--\cite{Korot2} for further results on the subject.

The direct and inverse spectral problems for Dirac operators with summable \emph{scalar} potentials were solved in \cite{MykDirac}, where an algorithm for reconstructing the operator from two spectra or from one spectrum and the norming constants was suggested. Later, using the technique that was suggested in \cite{sturm} for solving the inverse spectral problems for Sturm--Liouville operators with matrix-valued potentials, the case of Dirac operators with square-integrable matrix-valued potentials was considered in \cite{dirac1}. Therein, a complete description of the spectral data was given and a method for reconstructing the operator from the spectral data was suggested.

In this paper we extend the results of \cite{dirac1} to the case of Dirac operators with summable matrix-valued potentials. However, as compared with the special case $q\in L_2$, it turns out that more general one $q\in L_p$, $p\ge1$, meets some conceptual and technical difficulties.

In particular, the description of the spectral data in \cite{dirac1} involves a precise asymptotics of eigenvalues and norming matrices, which is not available for the operators with summable matrix-valued potentials. Namely, the proofs of the asymptotics in \cite{dirac1} were based on the results of \cite{trushzeros, myktrushzeros} that are applicable only for $p=2$. To prove similar statements for $p\ge1$ one may need analogues of the results of \cite{mykhrynzeros} covering matrix-valued functions, but to the best of the author's knowledge these are not known yet. Furthermore, it turns out that we cannot separate the asymptotics of eigenvalues and the asymptotics of norming matrices under the present assumptions on the potential.

Instead, as opposed to the case $p=2$, we formulate the description of the spectral data involving an auxiliary object which is the restriction of the Fourier transform of the spectral measure (see Theorem~\ref{Th1} and Definition~\ref{Def1} below and compare with Theorem~1.1 in \cite{dirac1}). The latter is easy to construct given the spectral data and turns out to contain all the information about the potential.

Further, as compared with the case $p=2$, while dealing with $p\ge1$ one also has to reconsider another condition involved in the description of the spectral data, i.e. the one on ranks of the norming matrices. Although the claim of the condition remains the same, technique of the proof that was used for $p=2$ fails for $p\ge1$. Instead, we introduce the approach based on the theory of Riesz bases. In particular, new our approach uses a vector analogue of well-known Kadec's $1/4$-theorem (see Appendix~\ref{AppRiesz}), which plays an auxiliary role in this paper but may be used while solving other problems having to do with vector-valued functions.

From the practical point of view, a method for reconstructing the operator from the spectral data appears the same as for operators with square-integrable potentials and essentially consists in solving certain integral equation (see Theorem~\ref{Th4} and further comments on it). The paper is theoretical, but the results can be used in practical applications where the inverse spectral problems for Dirac operators with matrix-valued potentials arise. For instance, inverse problems for quantum graphs, which are of practical importance in nanotechnology, microelectronics, etc. (see, e.g., [22]) in some cases can be reduced to the inverse problems for operators with matrix-valued potentials. As compared with the case considered in \cite{dirac1}, the results of the present paper allow stronger singularities of potentials which is important for applications.

The paper is organized as follows. In the reminder of this section we give a precise setting of the problem and formulate the main results. In Section~2 we provide some preliminaries: we introduce the Weyl--Titchmarsh function and establish the basic properties of the operators under consideration. In Sections~3 and 4, respectively, we solve the direct and inverse spectral problems. Several spaces used in this paper are introduced in Appendix~\ref{AppSpaces}. In Appendix~\ref{AppFact} we recall some facts on the factorization of integral operators. Finally, Appendix~\ref{AppRiesz} is devoted to some auxiliary facts on the theory of Riesz bases.

\subsection{Setting of the problem}\label{ProblSettingSubsect}

Firstly, we introduce the space of potentials for Dirac operators considered in this paper. Let $M_r$ denote the set of $r\times r$ matrices with complex entries, which we identify with the Banach algebra of linear operators $\mathbb C^r\to\mathbb C^r$ endowed with the standard norm. We write $I=I_r$ for the unit element of $M_r$ and $M_r^+$ for the set of all matrices $A\in M_r$ such that $A=A^*\ge0$.
We set
$$
\mathfrak Q_p:=L_p((0,1),M_r),\qquad p\in[1,\infty),
$$
and endow $\mathfrak Q_p$ with the norm $\|q\|_{\mathfrak Q_p}:=\left(\int_0^1\|q(s)\|^p\ \d s\right)^{1/p}$, $q\in\mathfrak Q_p$. The space $\mathfrak Q_p$ will serve as the \emph{space of potentials} for Dirac operators under consideration.

Let $q\in\mathfrak Q_p$, $p\in[1,\infty)$, denote
$$
\vartheta:=\frac{1}{\i}\begin{pmatrix}I&0\\0&-I\end{pmatrix},\qquad \bq:=\begin{pmatrix}0&q\\q^*&0\end{pmatrix}
$$
and consider the differential expression
\begin{equation}\label{DiffExpr}
\mathfrak t_q:=\vartheta\frac{\d}{\d x}+\bq
\end{equation}
on the domain
$$
D(\mathfrak t_q)=\{y:=(y_1, \ y_2)^\top \ | \ y_1,y_2\in W_1^1((0,1),\mathbb C^r) \},
$$
where $W_1^1((0,1),\mathbb C^r)$ is the Sobolev space.
The object of our investigation is the \emph{self-adjoint Dirac operator $T_q$} generated by the differential expression (\ref{DiffExpr}) and some separated boundary conditions:
$$
T_qy=\mathfrak t_q(y),
$$
$$
D(T_q)=\{y\in D(\mathfrak t_q) \ | \
\mathfrak t_q(y)\in L_2((0,1),\mathbb C^{2r}),\ y_1(0)=y_2(0), \ y_1(1)=y_2(1)\}.
$$
The function $q\in\mathfrak Q_p$ will be called the {\it potential} of the operator $T_q$.

The spectrum $\sigma(T_q)$ of the operator $T_q$ consists of countably many isolated real eigenvalues of finite multiplicity, accumulating only at $+\infty$ and $-\infty$. We denote by $\lambda_j(q)$, $j\in\mathbb Z$, the pairwise distinct eigenvalues of the operator $T_q$ labeled in increasing order so that $\lambda_0(q)\le0<\lambda_1(q)$, i.e.
$$
\sigma(T_q)=\{\lambda_j(q)\}_{j\in\mathbb Z},\qquad \lambda_0(q)\le0<\lambda_1(q).
$$

Further, denote by $m_q$ the Weyl--Titchmarsh function of the operator $T_q$ (see, e.g., \cite{ClarkGesz}). The function $m_q$ is a matrix-valued meromorphic Herglotz function (i.e. such that $\Im m_q(\lambda)\ge0$ whenever $\Im\lambda>0$), and $\{\lambda_j(q)\}_{j\in\mathbb Z}$ is the set of its poles. We set
\begin{equation}\label{NormMatrDef}
\alpha_j(q):=-\underset{\lambda=\lambda_j(q)}\res m_q(\lambda),\qquad j\in\mathbb Z,
\end{equation}
and call $\alpha_j(q)$ the {\it norming matrix} of the operator $T_q$ corresponding to the eigenvalue $\lambda_j(q)$. For every $j\in\mathbb Z$, the norming matrix $\alpha_j(q)$ is non-negative and multiplicity of the eigenvalue $\lambda_j(q)$ equals the rank of $\alpha_j(q)$.

The sequence $\mathfrak a_q:=((\lambda_j(q),\alpha_j(q)))_{j\in\mathbb Z}$ will be called the {\it spectral data} of the operator $T_q$, and the matrix-valued measure
\begin{equation}\label{SpectrMeasDef}
\mu_q:=\sum\limits_{j=-\infty}^{\infty}\alpha_j(q)\delta_{\lambda_j(q)}
\end{equation}
will be called its {\it spectral measure}. Here $\delta_\lambda$ is the Dirac delta-measure centered at the point $\lambda$. In particular, if $q=0$, then
\begin{equation}\label{Mu0Meas}
\mu_0=\sum\limits_{n=-\infty}^{\infty}I\delta_{\pi n}.
\end{equation}

We give a complete description of the class
$$
\mathfrak A_p:=\{\mathfrak a_q \ | \ q\in\mathfrak Q_p\}
$$
of the spectral data for the operators under consideration, which is equivalent to the description of the class $\mathfrak M_p:=\{\mu_q \ | \ q\in\mathfrak Q_p\}$ of the spectral measures. We then show that the spectral data of the operator $T_q$ determine the potential $q$ uniquely and suggest a method for reconstructing this potential from the spectral data.

\subsection{Main results}

We start from the description of the spectral data for the operators under consideration.
In what follows, $\mathfrak a$ will stand for an arbitrary sequence pretending to be the spectral data of some Dirac operator $T_q$, i.e. $\mathfrak a:=((\lambda_j,\alpha_j))_{j\in\mathbb Z}$, where $(\lambda_j)_{j\in\mathbb Z}$ is a strictly increasing sequence of real numbers such that $\lambda_0\le0<\lambda_1$ and $\alpha_j$, $j\in\mathbb Z$, are non-zero matrices in $M_r^+$.

Given the sequence $\mathfrak a$, denote by $\mu^{\mathfrak a}$ a matrix-valued measure given by
\begin{equation}\label{MuAdef}
\mu^{\mathfrak a}:=\sum\limits_{j=-\infty}^\infty \alpha_j\delta_{\lambda_j}
\end{equation}
and pretending to be the spectral measure.
With every measure $\mu:=\mu^{\mathfrak a}$ of the form (\ref{MuAdef}) we associate a $\mathbb C^r$-valued distribution
$$
(\mu,f):=\int\limits_{\mathbb R} f \ \d\mu,\qquad f\in\mathcal S^r,
$$
where $\mathcal S^r$ is the Schwartz space of rapidly decreasing $\mathbb C^r$-valued functions (see Appendix~\ref{AppSpaces}). Now we introduce a kind of the Fourier transform of $\mu^{\mathfrak a}$:

\begin{definition}\label{Def1}
For an arbitrary measure $\mu:=\mu^{\mathfrak a}$ of the form (\ref{MuAdef}) we denote by $\widehat\mu$ a $\mathbb C^r$-valued distribution given by the formula
$$
(\widehat\mu,f):=(\mu,\widehat f),\qquad f\in\mathcal S^r,
$$
where
\begin{equation}\label{Ftransform}
\widehat f(\lambda):=\int\limits_{-\infty}^\infty \e^{2\i\lambda s}f(s)\ \d s,
\qquad \lambda\in\mathbb R.
\end{equation}

We denote by $H_\mu$ the restriction of the distribution $\widehat\mu-\widehat\mu_0$ to the interval $[-1,1]$, i.e.
\begin{equation}\label{AccFirstDef}
(H_\mu,f):=(\widehat\mu-\widehat\mu_0,f),\qquad f\in\mathcal S^r,\ \ \supp f\subset[-1,1],
\end{equation}
where $\mu_0$ given by (\ref{Mu0Meas}) is the spectral measure of the free operator $T_0$.
\end{definition}

The distribution $H_\mu$, $\mu:=\mu^{\mathfrak a}$, plays an important role in establishing whether the sequence $\mathfrak a$ is the spectral data of some operator $T_q$. Namely, we partition the real axis into pairwise disjoint intervals $\Delta_n$, $n\in\mathbb Z$, by setting
$$
\Delta_n:=\left(\pi n-\frac{\pi}{2},\pi n+\frac{\pi}{2}\right].
$$
Then the following theorem gives a complete description of the class $\mathfrak A_p$ of the spectral data for Dirac operators under consideration:

\begin{theorem}\label{Th1}
In order that a sequence $\mathfrak a:=((\lambda_j,\alpha_j))_{j\in\mathbb Z}$ should belong to $\mathfrak A_p$, $p\ge1$, it is necessary and sufficient that the following conditions are satisfied:
\begin{itemize}
\item[$(B_1)$]$\sum_{\lambda_j\in\Delta_n}|\pi n-\lambda_j|=o(1)$ and
$\|I-\sum_{\lambda_j\in\Delta_n}\alpha_j\|=o(1)$ as $|n|\to\infty$;

\item[$(B_2)$]$\exists N_0\in\mathbb N \ \ \forall N\in \mathbb N:\ N>N_0\ \Rightarrow \ \sum_{n=-N}^{N}\sum_{\lambda_j\in\Delta_n}\rank \alpha_j=(2N+1)r$;

\item[$(B_3)$]the system of functions $\{ \e^{\i\lambda_jt}d\ | \ j\in\mathbb{Z},\ d\in\mathrm{Ran}\ \alpha_j \}$ is complete in $L_2((-1,1),\mathbb C^r)$;

\item[$(B_4)$]the distribution $H_\mu$, $\mu:=\mu^{\mathfrak a}$, belongs to $L_p((-1,1),M_r)$.
\end{itemize}
\end{theorem}

\begin{remark}
Thus, in particular, the distribution $H_\mu$, $\mu:=\mu^{\mathfrak a}$, turns out to be regular for all $\mathfrak a\in\mathfrak A_p$. In this case, $H_\mu(x)$ can be formally defined by the formula
\begin{equation}\label{Hseries}
H_\mu(x)=\sum_{n\in\mathbb Z}\left(\sum_{\lambda_j\in\Delta_n}\e^{2\i\lambda_jx}\alpha_j
-\e^{2\i\pi nx}I\right),\qquad x\in(-1,1),
\end{equation}
but the convergence of this series in $L_p((-1,1),M_r)$, $p\ge1$, is difficult to establish without knowing the precise asymptotics of eigenvalues and norming matrices. The proof for the case $p=2$ is given in \cite{dirac1}.
\end{remark}

Further, it turns out that the spectral data of the operator $T_q$ determine the potential $q$ uniquely:

\begin{theorem}\label{Th2}
For all $p\ge1$, the mapping $\mathfrak Q_p\owns q\mapsto\mathfrak a_q\in\mathfrak A_p$ is bijective.
\end{theorem}
This allows a possibility to reconstruct the operator from the spectral data.

As in \cite{sturm,dirac1}, we base our reconstruction algorithm on Krein's accelerant method:

\begin{definition}\label{Def2}
We say that a function $H\in L_1((-1,1),M_r)$ is an accelerant if $H(-x)=H(x)^*$ a.e. on $(-1,1)$ and for every $a\in(0,1]$ the integral equation
$$
f(x)+\int\limits_0^a H(x-t)f(t)\ \d t=0,\qquad x\in(0,a),
$$
has only zero solution in $L_2((0,a),\mathbb C^r)$.
We denote by $\mathfrak H_p$, $p\ge1$, the set of all accelerants belonging to $L_p((-1,1),M_r)$ and endow $\mathfrak H_p$ with the metric of the space $L_p((-1,1),M_r)$.
\end{definition}
Equivalently, it is known (see, e.g., \cite{MykDirac}) that an arbitrary function $H\in L_p((-1,1),M_r)$, $p\ge1$, belongs to $\mathfrak H_p$ if and only if the Krein equation
\begin{equation}\label{KreinEq}
R(x,t)+H(x-t)+\int\limits_0^x R(x,s)H(x-s)\ \d s=0,\qquad (x,t)\in\Omega,
\end{equation}
where $\Omega:=\{(x,t)\ | \ 0\le t\le x\le1\}$, is solvable in $G_p^+(M_r)$ (see Appendix~\ref{AppSpaces}). In this case, a solution of (\ref{KreinEq}) is unique and we denote it by $R_H(x,t)$, $(x,t)\in\Omega$.

We define the {\it Krein mapping} $\Theta:\mathfrak H_1\to\mathfrak Q_1$ by the formula
\begin{equation}\label{KreinMap}
[\Theta(H)](x):=\i R_H(x,0),\qquad x\in(0,1).
\end{equation}
This mapping provides a one-to-one correspondence between accelerants $H\in\mathfrak H_p$ and potentials $q\in\mathfrak Q_p$:
\begin{theorem}\label{Th3}
For all $p\ge1$, the Krein mapping $\Theta$ is a homeomorphism between the metric spaces $\mathfrak H_p$ and $\mathfrak Q_p$.
\end{theorem}
We use the Krein mapping to reconstruct the potential $q$ from the spectral data of the operator $T_q$:

\begin{theorem}\label{Th4}
Given an arbitrary sequence $\mathfrak a\in\mathfrak A_p$, set $\mu:=\mu^{\mathfrak a}$ and $H:=H_\mu$. Then $H\in\mathfrak H_p$ and $\mathfrak a=\mathfrak a_q$ for $q=\Theta(H)$.
\end{theorem}

Therefore, Theorems \ref{Th2} and \ref{Th4} provide an efficient method for reconstructing the Dirac operator $T_q$ from the spectral data. Namely, given an arbitrary sequence $\mathfrak a\in\mathfrak A_p$ being the spectral data of some Dirac operator $T_q$, construct the matrix-valued measure $\mu:=\mu^{\mathfrak a}$ by the formula (\ref{MuAdef}); set $H:=H_\mu$ by the formula (\ref{AccFirstDef}) or (\ref{Hseries}); substitute $H$ into the Krein equation (\ref{KreinEq}) and find $R_H$; find the potential $q=\Theta(H)$ by the formula (\ref{KreinMap}). That the potential $q$ is the one looked for follows from
the fact that the Dirac operator $T_q$ has the spectral data $\mathfrak a$ we have started with.

The method can be visualized by means of the following diagram:
$$
\mathfrak A_p
\owns
\mathfrak a
\overset{(1.5)}{\underset{s_1}\longrightarrow}
\mu:=\mu^{\mathfrak a}
\overset{(1.7)}{\underset{s_2}\longrightarrow}
H:=H_\mu
\overset{(1.10)}{\underset{s_3}\longrightarrow}
q=\Theta(H)
\in
\mathfrak Q_p.
$$
Here $s_j$ denotes the step number $j$. The steps $s_1$ and $s_2$ are trivial. The basic and non-trivial step is $s_3$, which requires solving the Krein equation (\ref{KreinEq}).

\begin{remark}
One can similarly consider the Dirac operators with general separated boundary conditions. However, the  description of the spectral data would be more complicated since the spectrum of the free operator $T_0$ (subject to $q=0$) is more involved in this case. The author plans to consider the operators with general (not necessarily separated) boundary conditions in a forthcoming paper.
\end{remark}

\section{Preliminary results}

Here we introduce the Weyl--Titchmarsh function and establish the basic properties of the operator $T_q$. The material of this section mainly follows \cite{dirac1} but forms the essential base for further considerations.

\subsection{The Weyl--Titchmarsh function of the operator $T_q$}

We start from constructing the Weyl--Titchmarsh function of the operator $T_q$.

Let $q\in\mathfrak Q_p$, $p\ge1$. Set
$$
a:=\frac{1}{\sqrt{2}}\begin{pmatrix}I,&-I\end{pmatrix}
$$
so that the boundary conditions $y_1(0)=y_2(0)$, $y_1(1)=y_2(1)$ can be written in the form
$$
ay(0)=ay(1)=0.
$$
The multiplier $\frac1{\sqrt2}$ provides the normalization $aa^*=I$.

Denote by $u_q(\cdot,\lambda)\in W_1^1((0,1),M_{2r})$, $\lambda\in\mathbb C$, a matrix-valued solution of the Cauchy problem
$$
\vartheta\frac{\d}{\d x}u+\bq u=\lambda u,\qquad u(0,\lambda)=I_{2r},
$$
and set
\begin{equation}\label{varphipsiDef}
\varphi_q(x,\lambda):=u_q(x,\lambda)\vartheta a^*,\qquad \psi_q(x,\lambda):=u_q(x,\lambda)a^*
\end{equation}
so that the $2r\times r$ matrix-valued functions $\varphi_q(x,\lambda)$ and $\psi_q(x,\lambda)$ solve the Cauchy problems
\begin{equation}\label{varphiCauchyPr}
\vartheta\frac{\d}{\d x}\varphi+\bq\varphi=\lambda\varphi,\qquad \varphi(0,\lambda)=\vartheta a^*,
\end{equation}
and
$$
\vartheta\frac{\d}{\d x}\psi+\bq\psi=\lambda\psi,\qquad \psi(0,\lambda)=a^*,
$$
respectively.

Next, define the $r\times r$ matrix-valued functions $s_q(\lambda)$ and $c_q(\lambda)$ by the formulae
\begin{equation}\label{scDef}
s_q(\lambda):=a\varphi_q(1,\lambda),\qquad c_q(\lambda):=a\psi_q(1,\lambda).
\end{equation}
Then the function
\begin{equation}\label{WeylFdef}
m_q(\lambda):=-s_q(\lambda)^{-1}c_q(\lambda)
\end{equation}
is called the {\it Weyl--Titchmarsh function} of the operator $T_q$.

Repeating the proofs in \cite{dirac1}, which were done for the case of square-integrable potential, one can prove the following lemma claiming the basic properties of just introduced objects:

\begin{lemma}\label{SCLemma}
\begin{itemize}
\item[(i)]For every $q\in\mathfrak Q_p$, $p\ge1$, there exists a unique matrix-valued function $K_q\in G_p^+(M_{2r})$ (see Appendix~\ref{AppSpaces}) such that for all $\lambda\in\mathbb C$ and $x\in(0,1)$,
\begin{equation}\label{VarphiTrasfOp}
\varphi_q(x,\lambda)=\varphi_0(x,\lambda)+\int\limits_0^x K_q(x,s)\varphi_0(s,\lambda)\ \d s,
\end{equation}
where
$$
\varphi_0(x,\lambda)=\frac1{\sqrt{2} \i}\begin{pmatrix}\e^{\i\lambda x}I\\
\e^{-\i\lambda x}I\end{pmatrix}
$$
is a solution of (\ref{varphiCauchyPr}) in the free case $\bq=0$; moreover, the mapping $\mathfrak Q_p\owns q\mapsto K_q\in G_p^+(M_{2r})$ is continuous;

\item[(ii)]the functions $s_q(\lambda)$ and $c_q(\lambda)$ are entire and allow the representations
\begin{equation}\label{Srepr}
s_q(\lambda)=(\sin\lambda)I+\int\limits_{-1}^1 \e^{\i\lambda t}g_q(t)\ \d t,\quad
\end{equation}
\begin{equation}\label{Crepr}
c_q(\lambda)=(\cos\lambda)I+\int\limits_{-1}^1 \e^{\i\lambda t}h_q(t)\ \d t,
\end{equation}
where $g_q$ and $h_q$ are some (depending on $q$) functions in $L_p((-1,1),M_r)$; moreover, the mappings $\mathfrak Q_p\owns q\mapsto g_q\in L_p((-1,1),M_r)$ and $\mathfrak Q_p\owns q\mapsto h_q\in L_p((-1,1),M_r)$ are continuous;

\item[(iii)]the operator functions $\lambda\mapsto s_q(\lambda)^{-1}$ and $\lambda\mapsto m_q(\lambda)$ are meromorphic in $\mathbb C$; moreover, $m_0(\lambda)=-(\cot\lambda)I$ and
\begin{equation}\label{mAsymp}
\|m_q(\lambda)+(\cot\lambda)I\|=\o(1)
\end{equation}
as $\lambda\to\infty$ within the domain $\mathcal O=\{z\in\mathbb C \ | \ \forall n\in\mathbb Z \ |z-\pi n|>1\}$.
\end{itemize}
\end{lemma}

\begin{proofsketch}
Repeating the proof in \cite{dirac1} and using the results of \cite{cauchy}, one can show that there exist unique matrix-valued functions $P^+,P^-\in G_p^+(M_{2r})$ such that for all $x\in(0,1)$ and $\lambda\in\mathbb C$,
$$
u_q(x,\lambda)=\e^{-\lambda x\vartheta}+\int_0^x P^+(x,s)\e^{-\lambda(x-2s)\vartheta}\ \d s+
\int_0^x P^-(x,s)\e^{\lambda(x-2s)\vartheta}\ \d s.
$$
Then, by virtue of definitions (\ref{varphipsiDef}) and (\ref{scDef}), straightforward calculations lead us to the representations (\ref{VarphiTrasfOp}), (\ref{Srepr}) and (\ref{Crepr}). Moreover, since the mappings $\mathfrak Q_p\owns q\mapsto P^\pm\in G_p^+(M_{2r})$ are continuous (see \cite{cauchy}), the mappings $\mathfrak Q_p\owns q\mapsto K_q\in G_p^+(M_{2r})$ and $\mathfrak Q_p\owns q\mapsto g_q,h_q\in L_p((-1,1),M_r)$ remain continuous. Thus we obtain parts (i) and (ii) of the present lemma.

To prove part (iii), observe that by virtue of the representations (\ref{Srepr}), (\ref{Crepr}) and Lemma~\ref{RefRiemannLebesgue} we have
$$
\lim\limits_{|\lambda|\to\infty} \e^{-|\Im\lambda|}\|s_q(\lambda)-(\sin\lambda)I\|=
\lim\limits_{|\lambda|\to\infty} \e^{-|\Im\lambda|}\|c_q(\lambda)-(\cos\lambda)I\|=0.
$$
Therefore, $s_q(\lambda)$ is invertible for all $\lambda\in\mathcal O$ large enough, so that $m_q$ is meromorphic and (\ref{mAsymp}) holds true.
\end{proofsketch}

It will follow that poles of the Weyl--Titchmarsh function $m_q$ are eigenvalues of the Dirac operator $T_q$. Given also the corresponding residues of $m_q$, it is possible to find the potential $q$.

\subsection{Basic properties of the operator $T_q$}

Here we establish the basic properties of the operator $T_q$.
We set
$$
\mathbb H:=L_2((0,1),\mathbb C^r)\times L_2((0,1),\mathbb C^r)
$$
and denote by $\mathscr I$ the identity operator in $\mathbb H$.
For an arbitrary $q\in\mathfrak Q_p$, $p\ge1$, and $\lambda\in\mathbb C$, denote by $\Phi_q(\lambda)$ the operator acting from $\mathbb C^r$ to $\mathbb H$ by the formula
$$
[\Phi_q(\lambda)c](x):=\varphi_q(x,\lambda)c.
$$
Taking into account (\ref{VarphiTrasfOp}), we obtain that for all $\lambda\in\mathbb C$,
\begin{equation}\label{PhiTransfOp}
\Phi_q(\lambda)=(\mathscr I+\mathscr K_q)\Phi_0(\lambda),
\end{equation}
where $\mathscr K_q$ is an integral operator in $\mathbb H$ with kernel $K_q$ (see Lemma~\ref{SCLemma}, part~(i)) and
\begin{equation}\label{Phi0Def}
[\Phi_0(\lambda)c](x)=\frac1{\sqrt{2} \i}\begin{pmatrix}\e^{\i\lambda x}I\\
\e^{-\i\lambda x}I\end{pmatrix}c.
\end{equation}

The following lemma claims basic properties of the operators $\Phi_q(\lambda)$. Particularly, the second part of the lemma is important:

\begin{lemma}\label{kerTLemma}
For all $q\in\mathfrak Q_p$, $p\ge1$, and $\lambda\in\mathbb C$,
\begin{itemize}
\item[(i)]
\begin{equation}\label{kerPhi}
\ker\Phi_q(\lambda)=\{0\},\qquad \Ran\Phi_q^*(\lambda)=\mathbb C^r,
\end{equation}
where $\Phi_q^*(\lambda):=[\Phi_q(\lambda)]^*$;
\item[(ii)]
\begin{equation}\label{kerT}
\ker(T_q-\lambda\mathscr I)=\Phi_q(\lambda)\ker s_q(\lambda).
\end{equation}
\end{itemize}
\end{lemma}

\begin{proof}
Since the operator $\mathscr I+\mathscr K_q$ is a homeomorphism of the space $\mathbb H$, part~(i) easily follows from (\ref{PhiTransfOp}) and (\ref{Phi0Def}).

To prove part~(ii), note that for all $c\in\ker s_q(\lambda)$ the function $f(x):=\varphi_q(x,\lambda)c$ verifies the equality $\vartheta f'+\mathbf qf=\lambda f$ and the boundary conditions $af(0)=af(1)=0$. Conversely, the generic solution of the problem
$$
\vartheta f'+\mathbf qf=\lambda f,\qquad af(0)=0,
$$
takes the form $f(x)=\varphi_q(x,\lambda)c$, $c\in\mathbb C^r$, while the boundary condition $af(1)=0$ means $c\in\ker s_q(\lambda)$. Therefore, the equality (\ref{kerT}) follows.
\end{proof}

Now, denote by $\lambda_j:=\lambda_j(q)$, $j\in\mathbb Z$, the pairwise distinct eigenvalues of the operator $T_q$ labeled in increasing order so that $\lambda_0\le0<\lambda_1$, and let $\alpha_j:=\alpha_j(q)$ be the corresponding norming matrices defined by (\ref{NormMatrDef}).
Then the basic properties of the operator $T_q$ are formulated in the following theorem:

\begin{theorem}\label{DirPropTh}
Let $q\in\mathfrak Q_p$, $p\ge1$. Then the following statements hold true:
\begin{itemize}
\item[$(i)$]the operator $T_q$ is self-adjoint;
\item[$(ii)$]the spectrum $\sigma(T_q)$ of the operator $T_q$ consists of isolated real eigenvalues of finite multiplicity and, moreover,
$$
\sigma(T_q)=\{\lambda\in\mathbb C \ | \ \ker s_q(\lambda)\neq\{0\}\};
$$
\item[$(iii)$]denote by $P_{q,j}$ the orthogonal projector onto $\ker(T_q-\lambda_j\mathscr I)$; then
\begin{equation}\label{ResIdent}
\sum_{j=-\infty}^\infty P_{q,j}=\mathscr I;
\end{equation}
\item[$(iv)$]for every $j\in\mathbb Z$ the norming matrix $\alpha_j$ is non-negative and
\begin{equation}\label{ProjForm}
P_{q,j}=\Phi_q(\lambda_j)\alpha_j\Phi_q^*(\lambda_j).
\end{equation}
\end{itemize}
\end{theorem}
The proof of Theorem~\ref{DirPropTh} repeats the proof in \cite{dirac1}. In particular, part~(ii) together with (\ref{WeylFdef}) implies that eigenvalues of the operator $T_q$ are poles of the Weyl--Titchmarsh function $m_q$.

By virtue of the relations (\ref{PhiTransfOp}), (\ref{ResIdent}) and (\ref{ProjForm}), the operators $\Phi_q(\lambda)$ and the function $K_q$ from Lemma~\ref{SCLemma} will play an important role in this investigation.

\section{Direct spectral problem}

Here we prove the necessity part of Theorem~\ref{Th1}: we show that for an arbitrary potential $q\in\mathfrak Q_p$, $p\ge1$, the spectral data of the operator $T_q$ satisfy the conditions $(B_1)$--$(B_4)$.
Throughout this section we use the abbreviations $\lambda_j:=\lambda_j(q)$ and $\alpha_j:=\alpha_j(q)$ for eigenvalues and norming matrices of the operator $T_q$, respectively.

\subsection{Condition $(B_1)$}

In this subsection we prove the following proposition:

\begin{proposition}\label{B1Lemma}
For an arbitrary potential $q\in\mathfrak Q_p$, $p\ge1$, the spectral data $\mathfrak a_q$ of the operator $T_q$ satisfy the condition $(B_1)$, i.e. the following asymptotics hold true:
\begin{equation}\label{LambdaAsymp}
\sum_{\lambda_j\in\Delta_n}|\pi n-\lambda_j|=\o(1),\qquad |n|\to\infty,
\end{equation}
and
\begin{equation}\label{AlphaAsymp}
\bigg\|I-\sum_{\lambda_j\in\Delta_n}\alpha_j\bigg\|=\o(1),\qquad |n|\to\infty,
\end{equation}
where $\Delta_n:=\left(\pi n-\frac{\pi}{2},\pi n+\frac{\pi}{2}\right]$.
\end{proposition}

As in \cite{dirac1}, the proof of Proposition~\ref{B1Lemma} is based on the claim that eigenvalues of the operator $T_q$ are zeros of certain entire function. However, as opposed to the case of square-integrable potential, we can establish only the rough asymptotics of eigenvalues and norming matrices.
Thus, we start the proof of Proposition~\ref{B1Lemma} from making the following remark:

\begin{remark}\label{B1tildeSrem}
As follows from Theorem~\ref{DirPropTh}, part (ii), the eigenvalues $(\lambda_j)_{j\in\mathbb Z}$ of the operator $T_q$ are zeros of the entire function
\begin{equation}\label{B1tildeSdef}
\widetilde s_q(\lambda):=\det s_q(\lambda).
\end{equation}
Since the function $s_q(\lambda)$ allows a representation (\ref{Srepr}), repeating the proofs in \cite{trushzeros} one can use Rouche's theorem to show that the set of zeros of the function $\widetilde s_q(\lambda)$ can be indexed (counting multiplicities) by numbers $n\in\mathbb Z$ so that the corresponding sequence $(\xi_n)_{n\in\mathbb Z}$ has the asymptotics
\begin{equation}\label{DetSzerosAsymp}
\xi_{nr+j}=\pi n+\o(1),\qquad j=1,\ldots,r,\ \ |n|\to\infty.
\end{equation}

Further, it also follows that the set of zeros of the entire function
$$
\widetilde c_q(\lambda):=\det c_q(\lambda)
$$
can be indexed (counting multiplicities) by numbers $n\in\mathbb Z$ so that the corresponding sequence $(\zeta_n)_{n\in\mathbb Z}$ has the asymptotics
\begin{equation}\label{DetCzerosAsymp}
\zeta_{nr+j}=\pi \left(n+\frac12\right)+\o(1),\qquad j=1,\ldots,r,\ \ |n|\to\infty.
\end{equation}
\end{remark}

Now the proof of Proposition~\ref{B1Lemma} is straightforward:

\begin{proofB1}
Since, by Remark~\ref{B1tildeSrem}, eigenvalues $(\lambda_j)_{j\in\mathbb Z}$ of the operator $T_q$ are zeros of the entire function $\widetilde s_q(\lambda)$, the asymptotics (\ref{LambdaAsymp}) directly follow from (\ref{DetSzerosAsymp}). Thus it only remains to prove (\ref{AlphaAsymp}).

For $n\in\mathbb Z$ denote
$$
\beta_n:=I-\sum_{\lambda_j\in\Delta_n}\alpha_j.
$$
It follows from the definition (\ref{NormMatrDef}) of $\alpha_j$ and from the asymptotics (\ref{LambdaAsymp}) of $(\lambda_j)_{j\in\mathbb Z}$ that there exists $n_0\in\mathbb N$ such that for all $n\in\mathbb Z$, $|n|>n_0$,
$$
\sum_{\lambda_j\in\Delta_n}\alpha_j=-\frac1{2\pi \i}\
\oint\limits_{|\lambda-\pi n|=1}m_q(\lambda)\ \d \lambda.
$$
Therefore, for all $n\in\mathbb Z$, $|n|>n_0$,
$$
\beta_n=\frac1{2\pi \i}\
\oint\limits_{|\lambda-\pi n|=1}(m_q(\lambda)+(\cot\lambda)I)\ \d\lambda.
$$
Now taking into account (\ref{mAsymp}), we observe that $\|\beta_n\|=\o(1)$, $|n|\to\infty$, as desired.
\end{proofB1}

\subsection{Conditions $(B_2)$ and $(B_3)$}

In this subsection we prove that the spectral data for the operators under consideration satisfy the conditions $(B_2)$ and $(B_3)$. We start from proving the condition $(B_2)$:

\begin{proposition}\label{B2Lemma}
For an arbitrary potential $q\in\mathfrak Q_p$, $p\ge1$, the spectral data $\mathfrak a_q$ of the operator $T_q$ satisfy the condition $(B_2)$, i.e. there exists $N_0\in\mathbb N$ such that for all natural $N>N_0$,
\begin{equation}\label{B2Eq}
\sum_{n=-N}^{N}\sum_{\lambda_j\in\Delta_n}\rank \alpha_j=(2N+1)r.
\end{equation}
\end{proposition}

The proof of Proposition~\ref{B2Lemma} involves a technique based on the theory of Riesz bases. However, before starting the essential part of the proof we have to establish some auxiliary results.
Namely, recalling Remark~\ref{B1tildeSrem} claiming that eigenvalues $(\lambda_j)_{j\in\mathbb Z}$ of the operator $T_q$ are zeros of the entire function $\widetilde s_q(\lambda)$, we need to prove the following lemma:

\begin{lemma}\label{B2multLemma}
Let $n_j$, $j\in\mathbb Z$, denote the multiplicity of zero $\lambda_j$ of the function $\widetilde s_q(\lambda)$.
Then there exists $N_1\in\mathbb N$ such that for all $\lambda_j\in\Delta_n$, $|n|>N_1$,
\begin{equation}\label{B2auxMultEq}
n_j=\rank\alpha_j.
\end{equation}
\end{lemma}

\begin{proof}
Firstly, observe that by virtue of the relations (\ref{kerPhi}), (\ref{kerT}) and (\ref{ProjForm}), we have
$$
\dim\ker s_q(\lambda_j)=\rank\alpha_j,\qquad j\in\mathbb Z.
$$
Since
$$
s_q(\lambda)=s_q(\lambda_j)+\Or(\lambda-\lambda_j),\qquad \lambda\to\lambda_j,
$$
one can easily find that
\begin{equation}\label{B1AuxEq2}
n_j\ge\rank\alpha_j,\qquad j\in\mathbb Z.
\end{equation}
Now let us show that there exists $N_1\in\mathbb N$ such that for all $\lambda_j\in\Delta_n$, $|n|>N_1$,
\begin{equation}\label{B2AuxEq2}
n_j\le\rank\alpha_j.
\end{equation}

Indeed, since the function $m_q(\lambda)$ is meromorphic in $\mathbb C$ and has a pole of first order at point $\lambda_j$ (see, e.g., \cite{geszHerg}), we find that
\begin{equation}\label{B2MauxRepr}
m_q(\lambda)=\frac{\alpha_j}{\lambda-\lambda_j}+g_j(\lambda),
\end{equation}
where the function $g_j(\lambda)$ is analytic in the neighborhood of $\lambda_j$. Denote by $Q_j$ an orthogonal projector onto $\Ran\alpha_j$ and set
$$
\widetilde Q_j(\lambda):=(I-Q_j)+(\lambda-\lambda_j)Q_j.
$$
Then it follows from (\ref{B2MauxRepr}) that the function
$$
\lambda\mapsto\widetilde Q_j(\lambda)m_q(\lambda)
$$
is bounded in the neighborhood of $\lambda_j$. Furthermore, observe that by virtue of the asymptotics (\ref{DetCzerosAsymp}) of zeros of $\det c_q(\lambda)$, the function $\lambda\mapsto c_q(\lambda)^{-1}$ is analytic in the neighborhood of $\lambda_j$ for large values of $|\lambda_j|$.

Now, since $m_q(\lambda)=-s_q(\lambda)^{-1}c_q(\lambda)$ (see (\ref{WeylFdef})), we find that the function
$$
\widetilde Q_j(\lambda)s_q(\lambda)^{-1}=-\widetilde Q_j(\lambda)m_q(\lambda)c_q(\lambda)^{-1}
$$
is bounded in the neighborhood of $\lambda_j$ for large values of $|\lambda_j|$. Therefore, since
$$
\det \widetilde Q_j(\lambda)s_q(\lambda)^{-1}=\frac{(\lambda-\lambda_j)^{\rank\alpha_j}}{\widetilde s_q(\lambda)},
$$
we observe that there exists $N_1\in\mathbb N$ such that for all $\lambda_j\in\Delta_n$, $|n|>N_1$, the inequality (\ref{B2AuxEq2}) holds true. Together with (\ref{B1AuxEq2}), this proves the lemma.
\end{proof}

Now, recalling the asymptotics (\ref{DetSzerosAsymp}) of zeros of $\widetilde s_q(\lambda)$, we arrive at the following corollary which was the purpose of Lemma~\ref{B2multLemma}:

\begin{corollary}\label{SumRankLemma}
There exists $N_2\in\mathbb N$ such that for all $n\in\mathbb Z$, $|n|>N_2$,
$$
\sum_{\lambda_j\in\Delta_n}\rank\alpha_j=r.
$$
\end{corollary}

Now we can proceed to the principal part of the proof involving the theory of Riesz bases. The approach is based on Lemma~\ref{B2MainLemma} below, which will be used in solving both the direct and inverse spectral problems. The proof of the Lemma~\ref{B2MainLemma} uses a vector analogue of well-known Kadec's $1/4$-theorem which is established in Appendix~\ref{AppRiesz}.

Thus, let $\mathfrak a:=((\lambda_j,\alpha_j))_{j\in\mathbb Z}$ be an {\it arbitrary} sequence where $(\lambda_j)_{j\in\mathbb Z}$ is a strictly increasing sequence of real numbers such that $\lambda_0\le0<\lambda_1$ and $\alpha_j$, $j\in\mathbb Z$, are non-zero matrices in $M_r^+$.
Since the matrices $\alpha_j$, $j\in\mathbb Z$, are self-adjoint and non-negative, for every $j\in\mathbb Z$ there exists a system of pairwise orthogonal vectors $\{v_{j,k}\}_{k=1}^{\rank\alpha_j}$ in $\mathbb C^r$ such that
\begin{equation}\label{AlphaVrel}
\alpha_j=\sum_{k=1}^{\rank\alpha_j} (\ \cdot\ |v_{j,k})v_{j,k}.
\end{equation}
Denote by $\epsilon_k$, $k=1,\ldots,r$, a standard orthonormal basis for $\mathbb C^r$. Then the following lemma holds true:

\begin{lemma}\label{B2MainLemma}
Let $\mathfrak a:=((\lambda_j,\alpha_j))_{j\in\mathbb Z}$ be an arbitrary sequence satisfying the asymptotics (\ref{LambdaAsymp}) and (\ref{AlphaAsymp}), and assume that there exists $N_0\in\mathbb N$ such that for all $n\in\mathbb Z$, $|n|>N_0$,
\begin{equation}\label{partA2Eq}
\sum_{\lambda_j\in\Delta_n}\rank\alpha_j=r.
\end{equation}
Then there exists $N_1\in\mathbb N$ such that for all natural $N>N_1$ the system $\mathcal E_N\cup\mathcal B_N$, where
\begin{equation}\label{B2ENdef}
\mathcal E_N:=\left\{\frac1{\sqrt2}\e^{\i\pi nt}\epsilon_s\ \big|\ n\in\mathbb Z,\ |n|\le N,
\quad s=1,\ldots,r \right\}
\end{equation}
and
\begin{equation}\label{B2BNdef}
\mathcal B_N:=\left\{\e^{\i\lambda_jt}v_{j,k}\ \big|\ \lambda_j\in\Delta_n,\ |n|>N,
\quad 1\le k\le\rank\alpha_j \right\},
\end{equation}
is a Riesz basis for the space $\mathcal H:=L_2((-1,1),\mathbb C^r)$ (see Appendix~\ref{AppRiesz}).
\end{lemma}

\begin{proof}
We prove the lemma by applying Theorem~\ref{KadecTh}.
As follows from the equality (\ref{partA2Eq}), for all $n\in\mathbb Z$, $|n|>N_0$, the system
$$
V_n:=\{v_{j,k}\ |\ \lambda_j\in\Delta_n,\ 1\le k\le\rank\alpha_j\}
$$
consists of exactly $r$ vectors. Moreover, since
\begin{equation}\label{B2AuxBetaN}
\|I-\beta_n\|=\o(1),\qquad |n|\to\infty,
\end{equation}
where
\begin{equation}\label{B2AuxBetaNdef}
\beta_n:=\sum_{\lambda_j\in\Delta_n}\alpha_j,
\end{equation}
we claim that $N_0$ can be taken so large that for all $n\in\mathbb Z$, $|n|>N_0$, the system $V_n$ forms a basis for the space $\mathbb C^r$.

For $n\in\mathbb Z$, $|n|>N_0$, we denote by $A_n$ the operator acting in $\mathbb C^r$ by the formula
$$
A_n v_{j,k}=\lambda_j v_{j,k},\qquad v_{j,k}\in V_n,\ \ \lambda_j\in\Delta_n,
$$
and set $B_n$ to be an operator transforming a standard orthonormal basis for $\mathbb C^r$ into the basis $V_n$.
It follows from the asymptotics (\ref{LambdaAsymp}) and (\ref{AlphaAsymp}) that there exists a natural $N_1>N_0$ such that
\begin{equation}\label{B2AuxAEst}
\|A_n-\pi nI\|<\ln 2,\qquad |n|>N_1.
\end{equation}
Moreover, let us show that
\begin{equation}\label{B2AuxBEst}
\sup\limits_{|n|>N_1} (\|B_n\|+\|{B_n}^{-1}\|)<\infty.
\end{equation}

Indeed, since
$$
\sum_{v_{j,k}\in V_n}(\ \cdot\ |v_{j,k})v_{j,k}=\beta_n,
$$
where $\beta_n$ is self-adjoint non-negative matrix given by (\ref{B2AuxBetaNdef}), we observe that for all $|n|>N_1$ the vectors ${\beta_n}^{-1/2}v_{j,k}$, where $v_{j,k}\in V_n$, form an orthonormal basis for $\mathbb C^r$. Therefore, $B_n={\beta_n}^{1/2}U_n$, where $U_n$ is a unitary matrix in $M_r$. Since, by virtue of (\ref{B2AuxBetaN}), we have
$$
\sup\limits_{|n|>N_1} (\|{\beta_n}^{1/2}\|+\|{\beta_n}^{-1/2}\|)<\infty,
$$
(\ref{B2AuxBEst}) follows.

Therefore, taking into account (\ref{B2AuxAEst}) and (\ref{B2AuxBEst}), we find from Theorem~\ref{KadecTh} that for all natural $N>N_1$ the system $\mathcal E_N\cup\mathcal B_N$ forms a Riesz basis for the space $\mathcal H$, as desired.
\end{proof}

Now we use Lemma~\ref{B2MainLemma} and Corollary~\ref{SumRankLemma} to prove Proposition~\ref{B2Lemma}:

\begin{proofB2}
Since the norming matrices $\alpha_j$, $j\in\mathbb Z$, are self-adjoint and non-negative, for every $j\in\mathbb Z$ there exists a system of pairwise orthogonal vectors $\{v_{j,k}\}_{k=1}^{\rank\alpha_j}$ in $\mathbb C^r$ such that (\ref{AlphaVrel}) holds true.

Observe that by virtue of Proposition~\ref{B1Lemma} and Corollary~\ref{SumRankLemma}, the conditions of Lemma~\ref{B2MainLemma} are satisfied. Hence, it follows from Lemma~\ref{B2MainLemma} that there exists $N_1\in\mathbb N$ such that for all natural $N>N_1$, the system $\mathcal E_N\cup\mathcal B_N$ (see (\ref{B2ENdef}) and (\ref{B2BNdef})) is a Riesz basis for the space $\mathcal H:=L_2((-1,1),\mathbb C^r)$. Therefore, Proposition~\ref{B2Lemma} will be proved if we show that the system
\begin{equation}\label{B2Cinfty}
\mathcal B_0:=\left\{\e^{\i\lambda_jt}v_{j,k}\ \big|\ j\in\mathbb Z,\ \ 1\le k\le\rank\alpha_j \right\}
\end{equation}
also forms a basis for $\mathcal H$. Indeed, since $\mathcal B_N\subset\mathcal B_0$ and
$$
\mathcal B_0\setminus\mathcal B_N=\left\{\e^{\i\lambda_jt}v_{j,k}\ \big|\ \lambda_j\in\Delta_n,\ |n|\le N,
\quad 1\le k\le\rank\alpha_j \right\},
$$
we then obtain that for all natural $N>N_1$, the finite systems $\mathcal B_0\setminus\mathcal B_N$ and $\mathcal E_N$ consist of the same number of elements. Obviously, this implies the claim of Proposition~\ref{B2Lemma}.

Thus let us prove that $\mathcal B_0$ is a basis for $\mathcal H$.
Indeed, since the operators
$$
P_{q,j}:=\Phi_q(\lambda_j)\alpha_j\Phi_q^*(\lambda_j),\qquad j\in\mathbb Z,
$$
form a system of pairwise orthogonal projectors in the space $\mathbb H$ (see Theorem~\ref{DirPropTh}) and
$$
\sum_{j=-\infty}^\infty P_{q,j}=\mathscr I,
$$
we find that the system
$$
\boldsymbol{\mathcal A}_q:=\{\Phi_q(\lambda_j)v_{j,k}\ |\ j\in\mathbb Z,\
\ 1\le k\le\rank\alpha_j\}
$$
is a basis for the space $\mathbb H$. Since $\Phi_q(\lambda)=(\mathscr I+\mathscr K_q)\Phi_0(\lambda)$ and the mapping $\mathscr I+\mathscr K_q$ is a homeomorphism of $\mathbb H$, we observe that the system
$$
\boldsymbol{\mathcal A}_0:=\{\Phi_0(\lambda_j)v_{j,k}\ |\ j\in\mathbb Z,\
\ 1\le k\le\rank\alpha_j\}
$$
remains a basis for $\mathbb H$.
Now introduce the unitary mapping $V:\mathcal H\to\mathbb H$ acting by the formula
$$
(Vf)(t)=\begin{pmatrix}f(-t),&f(t)\end{pmatrix}^\top,\qquad t\in(0,1),
$$
and note that $V$ maps the function $\e^{\i\lambda_jt}v_{j,k}$ to $\Phi_0(\lambda_j)v_{j,k}$ for all $j\in\mathbb Z$, $1\le k\le\rank\alpha_j$. Therefore, the system $\mathcal B_0$ is a basis for the space $\mathcal H$, as desired.
\end{proofB2}

In particular, from the proof of Proposition~\ref{B2Lemma} we also obtain the following:

\begin{corollary}\label{B3Lemma}
For an arbitrary potential $q\in\mathfrak Q_p$, $p\ge1$, the spectral data $\mathfrak a_q$ of the operator $T_q$ satisfy the condition $(B_3)$, i.e. the system
\begin{equation}\label{XiB4}
\mathcal X:=\{ \e^{\i\lambda_jt}d\ | \ j\in\mathbb{Z},\ d\in\mathrm{Ran}\ \alpha_j \}
\end{equation}
is complete in $L_2((-1,1),\mathbb C^r)$.
\end{corollary}

\begin{proof}
It follows from the proof of Proposition~\ref{B2Lemma} that the system $\mathcal B_0$ given by (\ref{B2Cinfty}) is a basis for the space $\mathcal H:=L_2((-1,1),\mathbb C^r)$. Therefore, we immediately obtain that $\mathcal X^\bot=(\lin\mathcal B_0)^\bot=\{0\}$, as desired. Here $\lin\mathcal B_0$ denotes the linear span of $\mathcal B_0$.
\end{proof}

\subsection{The Krein mapping. Proof of Theorem~\ref{Th3}}

Thus, it remains only to prove that the spectral data for the operators under consideration satisfy the condition $(B_4)$. In order to do this, we need to look into some properties of the Krein mapping $\Theta$ given by (\ref{KreinMap}). In particular, here we prove Theorem~\ref{Th3} claiming that for all $p\ge1$ the mapping $\Theta$ is a homeomorphism between the space $\mathfrak H_p$ of accelerants and the space $\mathfrak Q_p$ of potentials.

We start from the following remark:

\begin{remark}\label{ThetaContRem}
It obviously follows from the results of Appendix~\ref{AppFact} that for all $p\ge1$, the mapping $H\mapsto\Theta(H):=\i R_H(\cdot,0)$ acts continuously from $\mathfrak H_p$ to $\mathfrak Q_p$.
\end{remark}

Now, for an arbitrary accelerant $H\in\mathfrak H_p$ denote by $\mathscr F_H$ a self-adjoint integral operator in $\mathbb H$ with kernel
\begin{equation}\label{FH}
F_H(x,t):=\frac{1}{2}\begin{pmatrix}H\left(\frac{x-t}{2}\right)&
H\left(\frac{x+t}{2}\right)\\
H\left(-\frac{x+t}{2}\right)&
H\left(-\frac{x-t}{2}\right)\end{pmatrix},
\qquad 0\le x,t\le1.
\end{equation}
Then the following proposition explains a natural connection between an accelerant $H$ and the corresponding potential $q:=\Theta(H)$:

\begin{proposition}\label{DirectFactTh}
For an arbitrary accelerant $H\in\mathfrak H_p$,
\begin{itemize}
\item[(i)]there exists a unique function $L_H\in G_p^+(M_{2r})$ such that
\begin{equation}\label{FHfact}
\mathscr I+\mathscr F_H=(\mathscr I+\mathscr L_H)^{-1}(\mathscr I+{\mathscr L_H}^*)^{-1},
\end{equation}
where $\mathscr L_H\in\mathscr G_p^+(M_{2r})$ is an integral operator in $\mathbb H$ with kernel $L_H$;
\item[(ii)]$L_H=K_q$ for $q=\Theta(H)$.
\end{itemize}
\end{proposition}

\begin{proof}
Part {\it (i)} of the present proposition directly follows from Lemma~\ref{FactLemma1}.
To prove part {\it (ii)}, note that for all $p\ge1$ the set $\mathfrak H_p$ is open in $L_p((-1,1),M_r)$, and thus the set $\mathfrak H_1\cap C([-1,1],M_r)$ of continuous accelerants is dense everywhere in $\mathfrak H_p$. Therefore, there exists a sequence $(H_n)_{n\in\mathbb N}$, $H_n\in\mathfrak H_1\cap C([-1,1],M_r)$, such that
$$
\lim\limits_{n\to\infty}\|H-H_n\|_{\mathfrak H_p}=0.
$$
Then it follows from \cite{dirac1} that for all $n\in\mathbb N$,
\begin{equation}\label{LHKqEq1}
L_{H_n}=K_{q_n},\qquad q_n=\Theta(H_n).
\end{equation}
Now, by virtue of Theorem~\ref{FactTh2}, we find that the mapping $\mathfrak H_p\owns H\mapsto L_H\in G_p^+(M_{2r})$ is continuous. From the other side, it follows from Lemma~\ref{SCLemma} that so is the mapping $\mathfrak Q_p\owns q\mapsto K_q\in G_p^+(M_{2r})$. Furthermore, as follows from Remark~\ref{ThetaContRem}, the mapping $\Theta$ acts continuously from $\mathfrak H_p$ to $\mathfrak Q_p$. Hence, passing to the limit in (\ref{LHKqEq1}) we obtain that
$$
L_{H}=K_{q},\qquad q=\Theta(H),
$$
as desired.
\end{proof}

Now we are ready to prove Theorem~\ref{Th3}. The proof also uses the results of \cite{5auth} and \cite{dirac1}:

\begin{proofTh3}
As follows from Remark~\ref{ThetaContRem}, the mapping $\Theta$ acts continuously from $\mathfrak H_p$ to $\mathfrak Q_p$ for all $p\ge1$. Therefore, it remains to prove that $\Theta$ is invertible and that the inverse mapping $\Theta^{-1}$ acts continuously from $\mathfrak Q_p$ to $\mathfrak H_p$.

Denote by $\mathfrak H_0$ the set of all accelerants that are continuous on $[-1,1]\setminus\{0\}$ having a jump discontinuity at the origin, and set
$$
\Theta_0:=\Theta|_{\mathfrak H_0}.
$$
Then it is proved in \cite{5auth} that $\Theta_0$ maps $\mathfrak H_0$ onto $C([0,1],M_r)$ one-to-one. Since $\mathfrak H_0$ is dense in $\mathfrak H_p$ and $C([0,1],M_r)$ is dense in $\mathfrak Q_p$ for all $p\ge1$, the present theorem will be proved if we show that ${\Theta_0}^{-1}$ can be extended to a continuous mapping from $\mathfrak Q_p$ to $\mathfrak H_p$.

Recalling the results of \cite{dirac1}, we find that the mapping ${\Theta_0}^{-1}$ can be represented in the following way. For an arbitrary $q\in C([0,1],M_r)$ denote by $\mathscr K_q$ an integral operator in $\mathbb H$ with kernel $K_q$ (see Lemma~\ref{SCLemma}) and set
\begin{equation}\label{FIntDef}
\mathscr F^q:=(\mathscr I+\mathscr K_q)^{-1}
(\mathscr I+{\mathscr K_q}^*)^{-1}-\mathscr I.
\end{equation}
Let $F^q$ be the kernel of $\mathscr F^q$.
Since the mappings $\mathfrak Q_p\owns q\mapsto K_q \in G_p^+(M_{2r})$ and $\mathscr G_p^+(M_{2r})\owns\mathscr K\mapsto(\mathscr I+\mathscr K)^{-1}-\mathscr I\in\mathscr G_p^+(M_{2r})$ are continuous (see Lemmas \ref{SCLemma} and \ref{KinvLemma}, respectively), it follows that the mapping $\mathfrak Q_p\owns q\mapsto F^q \in G_p(M_{2r})$ is continuous as well.

As follows from Proposition~\ref{DirectFactTh}, if $q=\Theta(H)$ for some $H\in\mathfrak H_p$, then $F^q=F_H$. By virtue of the formula (\ref{FH}) for $F_H$, the inverse of the mapping $\mathfrak H_p\owns H\mapsto F_H\in G_p(M_{2r})$ is a restriction of certain mapping $\eta:G_p(M_{2r})\to L_p((-1,1),M_r)$ that can be easily written: write $F\in G_p(M_{2r})$ in the block-diagonal form
$$
F=\frac12\begin{pmatrix}F_{11}&F_{12}\\
F_{21}&F_{22}\end{pmatrix},
$$
where $F_{ks}\in G_p(M_r)$, $k,s\in\{1,2\}$, and set
\begin{equation}\label{AFdef}
[\eta(F)](x):=\left\{\begin{array}{cl}F_{21}(-2x-1,1),&-1\le x\le-\frac12,\\
F_{11}(2x+1,1),&-\frac12<x\le0,\\
F_{22}(-2x+1,1),&0<x\le\frac12,\\
F_{12}(2x-1,1),&\frac12<x\le1.\end{array}\right.
\end{equation}

Now define the mapping $\Upsilon$ acting from $C([0,1],M_r)$ to $L_p((-1,1),M_r)$ by the formula
$$
\Upsilon(q):=\eta(F^q),\qquad q\in C([0,1],M_r).
$$
Since the mappings $$\mathfrak Q_p\owns q\mapsto F^q \in G_p(M_{2r})$$ and $$\eta:G_p(M_{2r})\to L_p((-1,1),M_r)$$ are continuous, we obtain that $\Upsilon$ acts continuously from $\mathfrak Q_p$ to $L_p((-1,1),M_r)$.

It follows from \cite{dirac1} that
$$
\Upsilon(q)={\Theta_0}^{-1}(q)
$$
for all $q\in C([0,1],M_r)$. Therefore, we find that the mapping ${\Theta_0}^{-1}$ can be extended to a continuous one from $\mathfrak Q_p$ to $\mathfrak H_p$, as desired.
\end{proofTh3}

Theorem~\ref{Th3} will find its application in the next subsection, while Proposition~\ref{DirectFactTh} will also play an important role in solving the inverse spectral problem.

\subsection{Condition $(B_4)$}

The main result of this subsection is the following theorem:

\begin{theorem}\label{B3Th}
Let $q\in\mathfrak Q_p$, $p\ge1$, set $\mu:=\mu_q$ and $H:=H_\mu$ (see the definitions (\ref{SpectrMeasDef}) and (\ref{AccFirstDef}), respectively). Then $H\in\mathfrak H_p$ and $q=\Theta(H)$.
\end{theorem}
In particular, from Theorem~\ref{B3Th} we immediately obtain the following:

\begin{corollary}\label{B4Lemma}
For an arbitrary potential $q\in\mathfrak Q_p$, $p\ge1$, the spectral data $\mathfrak a_q$ of the operator $T_q$ satisfy the condition $(B_4)$, i.e. $H_\mu\in L_p((-1,1),M_r)$, $\mu:=\mu_q$.
\end{corollary}

We prove Theorem~\ref{B3Th} by a limiting procedure using the results of \cite{dirac1}. Before starting the proof, we need to establish two auxiliary technical lemmas.

\begin{lemma}\label{MLimLemma}
Let $q\in\mathfrak Q_p$, $p\ge1$. Assume that for all $n\in\mathbb N$, $q_n\in C([0,1], M_r)$ and $\lim_{n\to\infty}\|q-q_n\|_{\mathfrak Q_p}=0$. Then there exist $N_0\in\mathbb N$ and $n_0\in\mathbb N$ such that for any natural $N>N_0$ and $n>n_0$ the functions $m_q$ and $m_{q_n}$ have no poles on the circle $K_N:=\{\lambda\in\mathbb C\ | \ |\lambda|=\pi N+\pi/6\}$.
Moreover, for all natural $N>N_0$,
\begin{equation}\label{MLimLemmaEq}
\lim\limits_{n\to\infty} \ \sup\limits_{\lambda\in K_N} \|m_q(\lambda)-m_{q_n}(\lambda)\|=0.
\end{equation}
\end{lemma}

\begin{proof}
It follows from Lemma~\ref{SCLemma} that
\begin{equation}\label{SSnformulas}
s_q(\lambda)=(\sin\lambda)I+\mathcal F(\lambda)g_q,\qquad
s_{q_n}(\lambda)=(\sin\lambda)I+\mathcal F(\lambda)g_{q_n},
\end{equation}
where $g_q,g_{q_n}\in L_p((-1,1),M_r)$ and $\mathcal F(\lambda)$ is the operator $L_p((-1,1),M_r)\to M_r$ acting by the formula
$\mathcal F(\lambda)f:=\int_{-1}^1 \e^{\i\lambda t}f(t)\ \d t$. Moreover, since $q_n\to q$ in the metric of the space $\mathfrak Q_p$, from Lemma~\ref{SCLemma} we also obtain that $\lim_{n\to\infty}\|g_{q_n}-g_q\|_{L_1}=0$.

Similarly,
$$
c_q(\lambda)=(\cos\lambda)I+\mathcal F(\lambda)h_q,\qquad
c_{q_n}(\lambda)=(\cos\lambda)I+\mathcal F(\lambda)h_{q_n},
$$
where $h_q,h_{q_n}\in L_p((-1,1),M_r)$ and $\lim_{n\to\infty}\|h_{q_n}-h_q\|_{L_1}=0$.

Since $\|g_{q_n}-g_q\|_{L_1}\to0$ and $\|h_{q_n}-h_q\|_{L_1}\to0$ as $n\to\infty$, we find that for every $N\in\mathbb N$,
\begin{equation}\label{SSnCCn}
\lim\limits_{n\to\infty}\ \sup\limits_{\lambda\in K_N}\|s_q(\lambda)-s_{q_n}(\lambda)\|=0,\qquad
\lim\limits_{n\to\infty}\ \sup\limits_{\lambda\in K_N}\|c_q(\lambda)-c_{q_n}(\lambda)\|=0.
\end{equation}
Now let us establish some estimates for $\|s_q(\lambda)^{-1}\|$ and $\|s_{q_n}(\lambda)^{-1}\|$ as $\lambda\in K_N$.

By virtue of Lemma~\ref{RefRiemannLebesgue}, we obtain that there exists $N_1\in\mathbb N$ such that for all natural $N>N_1$,
\begin{equation}\label{RiemannLebesgueInequality}
\e^{-|\Im\lambda|}\|\mathcal F(\lambda)g_q\|<\frac18,\qquad \lambda\in K_N.
\end{equation}
In the meantime, using the expansion of $\sin\lambda$ as an infinite product we note that $|\sin\lambda|\ge|\sin|\lambda||$, $\lambda\in\mathbb C$, and thus $|\sin\lambda|\ge\frac12$ as $\lambda\in K_N$. Moreover, as $|\Im\lambda|>\ln2$ we arrive at the estimate
$$
|\sin\lambda|=\frac12|\e^{\i\lambda}-\e^{-\i\lambda}|\ge\frac14\ \e^{|\Im\lambda|},
$$
while as $|\Im\lambda|\le \ln2$, $\lambda\in K_N$, we have
$$
|\sin\lambda|\ge\frac12\ge\frac14\ \e^{|\Im\lambda|}.
$$
Thus, $|\sin\lambda|\ge\frac14\ \e^{|\Im\lambda|}$ as $\lambda\in K_N$. Using this estimate together with (\ref{RiemannLebesgueInequality}), we obtain from (\ref{SSnformulas}) that
\begin{equation}\label{SInvEst}
\|s_q(\lambda)^{-1}\| \le \frac{|\sin\lambda|^{-1}}{1-|\sin\lambda|^{-1}\cdot\|\mathcal F(\lambda)g_q\|}\le4,\qquad \lambda\in K_N,\ N>N_1.
\end{equation}

Similarly, since $\|g_{q_n}-g_q\|_{L_1}\to0$ as $n\to\infty$ and $\|\mathcal F(\lambda)(g_{q_n}-g_q)\|\le \e^{|\Im\lambda|}\|g_{q_n}-g_q\|_{L_1}$, we obtain that there exist $n_0\in\mathbb N$ and $N_2\in\mathbb N$ such that for all natural $n>n_0$ and $N>N_2$,
$$
\e^{-|\Im\lambda|}\|\mathcal F(\lambda)g_{q_n}\| \le \e^{-|\Im\lambda|}\|\mathcal F(\lambda)g_q\|+\e^{-|\Im\lambda|}\|\mathcal F(\lambda)(g_q-g_{q_n})\|<\frac18
$$
as $\lambda\in K_N$. Therefore, for all natural $n>n_0$,
\begin{equation}\label{SnInvEst}
\|s_{q_n}(\lambda)^{-1}\| \le \frac{|\sin\lambda|^{-1}}{1-|\sin\lambda|^{-1}\cdot\|\mathcal F(\lambda)g_{q_n}\|}\le4,
\quad \lambda\in K_N,\ N>N_2.
\end{equation}

Hence we conclude that the functions $m_q$ and $m_{q_n}$, $n>n_0$, have no poles on $K_N$ as $N>N_0:=\max\{N_1,N_2\}$. Moreover, it follows from the definitions of $m_q$ and $m_{q_n}$ that
\begin{align*}
\|m_q(\lambda)-m_{q_n}(\lambda)\|=\|-s_q(\lambda)^{-1}c_q(\lambda)+s_{q_n}(\lambda)^{-1}c_{q_n}(\lambda)\|
\\
\le\|s_{q_n}(\lambda)^{-1}\|\cdot\|c_q(\lambda)-c_{q_n}(\lambda)\|+
\|c_q(\lambda)\|\cdot\|s_q(\lambda)^{-1}-s_{q_n}(\lambda)^{-1}\|.
\end{align*}
Since
$$
\|s_q(\lambda)^{-1}-s_{q_n}(\lambda)^{-1}\|\le
\|s_{q_n}(\lambda)^{-1}\|\cdot\|s_q(\lambda)-s_{q_n}(\lambda)\|\cdot\|s_q(\lambda)^{-1}\|,
$$
taking into account (\ref{SSnCCn}), (\ref{SInvEst}) and (\ref{SnInvEst}) we arrive at the equality (\ref{MLimLemmaEq}).
\end{proof}

We use the preceding lemma to obtain the following auxiliary result:

\begin{lemma}\label{WeakConvProp}
Let $q\in\mathfrak Q_p$, $p\ge1$. Assume that for all $n\in\mathbb N$, $q_n\in C([0,1], M_r)$ and $\lim_{n\to\infty}\|q-q_n\|_{\mathfrak Q_p}=0$. Then for all $f\in\mathcal S^r$ (see Appendix~\ref{AppSpaces}) such that $\supp f\subset[-1,1]$,
$$
\lim\limits_{n\to\infty}(H_{\mu_{q_n}},f)=(H_{\mu_q},f),
$$
where $\mu_{q_n}$ and $\mu_q$ are the spectral measures of the operators $T_{q_n}$ and $T_q$, respectively.
\end{lemma}

\begin{proof}
Let $(\alpha_j,\lambda_j)_{j\in\mathbb Z}$ and $(\alpha_{j,n},\lambda_{j,n})_{j\in\mathbb Z}$ be the spectral data of the operators $T_q$ and $T_{q_n}$, respectively. Then for all $f\in\mathcal S^r$, $\supp f\subset[-1,1]$,
$$
(H_{\mu_{q_n}},f)=
\sum\limits_{j=-\infty}^\infty \alpha_{j,n} \widehat f(\lambda_{j,n})-\sum\limits_{k=-\infty}^\infty  \widehat f(\pi k),
$$
where $\widehat f(\lambda)$ is defined by (\ref{Ftransform}).
Since $\widehat f\in\mathcal S^r$ and for all $n\in\mathbb N$ the norms $\|\alpha_{j,n}\|$, $j\in\mathbb Z$, are uniformly bounded (see (\ref{AlphaAsymp})), it suffices to show that there exists $N_0\in\mathbb N$ such that for all natural $N>N_0$,
\begin{equation}\label{LimEq1}
\lim\limits_{n\to\infty}\sum\limits_{|\lambda_{j,n}|<\pi N+\frac\pi6} \alpha_{j,n} \widehat f(\lambda_{j,n}) =
\sum\limits_{|\lambda_j|<\pi N+\frac\pi6} \alpha_j \widehat f(\lambda_j).
\end{equation}

Note that for fixed $N\in\mathbb N$ both the left and right parts of (\ref{LimEq1}) depend only on the values of $\widehat f$ on some (large enough) finite interval. Since the function $\widehat f\in\mathcal S^r$ can be uniformly approximated by ``polynomials'' from $\mathcal P^r$ (see Appendix~\ref{AppSpaces}) on any finite interval, (\ref{LimEq1}) will be proved if we show that
$$
\lim\limits_{n\to\infty}\sum\limits_{|\lambda_{j,n}|<\pi N+\frac\pi6} \alpha_{j,n} P(\lambda_{j,n}) =
\sum\limits_{|\lambda_j|<\pi N+\frac\pi6} \alpha_j P(\lambda_j)
$$
for an arbitrary $P\in\mathcal P^r$.

Since the functions $m_q$ and $m_{q_n}$ have only poles of first order at points $\lambda_j$ and $\lambda_{j,n}$, respectively, by virtue of the asymptotics (\ref{LambdaAsymp}) and Lemma~\ref{MLimLemma} we find that
$$
\sum\limits_{|\lambda_{j,n}|<\pi N+\frac\pi6} \alpha_{j,n} P(\lambda_{j,n})=\frac{1}{2\pi \i}\oint\limits_{K_N}m_{q_n}(\lambda)P(\lambda)\ \d\lambda
$$
and
$$
\sum\limits_{|\lambda_j|<\pi N+\frac\pi6} \alpha_j P(\lambda_j)=\frac{1}{2\pi \i}\oint\limits_{K_N}m_q(\lambda)P(\lambda)\ \d\lambda
$$
for large values of $N$, where $K_N:=\{\lambda\in\mathbb C\ | \ |\lambda|=\pi N+\pi/6\}$. Hence it is enough to prove that there exists $N_0\in\mathbb N$ such that for all natural $N>N_0$,
$$
\lim\limits_{n\to\infty}\frac{1}{2\pi \i}\oint\limits_{K_N}m_{q_n}(\lambda)P(\lambda)\ \d\lambda =
\frac{1}{2\pi \i}\oint\limits_{K_N}m_q(\lambda)P(\lambda)\ \d\lambda.
$$
But this claim follows directly from (\ref{MLimLemmaEq}), and thus the proof is complete.
\end{proof}

Now we can use Lemma~\ref{WeakConvProp} to prove Theorem~\ref{B3Th}:

\begin{proofB3}
Let $q\in\mathfrak Q_p$, $p\ge1$. Assume that for all $n\in\mathbb N$, $q_n\in C([0,1],M_r)$ and $\lim_{n\to\infty}\|q-q_n\|_{\mathfrak Q_p}=0$. Let $\mu_q$ and $\mu_{q_n}$, $n\in\mathbb N$, be the spectral measures of Dirac operators $T_q$ and $T_{q_n}$, respectively. Then from Lemma~\ref{WeakConvProp} we obtain that for all $f\in\mathcal S^r$ such that $\supp f\subset[-1,1]$,
\begin{equation}\label{B4AuxEq2}
\lim\limits_{n\to\infty}(H_{\mu_{q_n}},f)=(H_{\mu_q},f).
\end{equation}
It follows from \cite{dirac1} that for all $n\in\mathbb N$, $H_{\mu_{q_n}}\in \mathfrak H_2$ and that $H_{\mu_{q_n}}=\Theta^{-1}(q_n)$. Hence (\ref{B4AuxEq2}) reads
$$
\lim\limits_{n\to\infty}(\Theta^{-1}(q_n),f)=(H_{\mu_q},f).
$$
From the other side, since $\Theta^{-1}$ acts continuously from $\mathfrak Q_p$ to $\mathfrak H_p$, we observe that for all $f\in\mathcal S^r$, $\supp f\subset[-1,1]$,
$$
\lim\limits_{n\to\infty}(\Theta^{-1}(q_n),f)=(\Theta^{-1}(q),f).
$$
Thus we obtain that for all $f\in\mathcal S^r$, $\supp f\subset[-1,1]$,
$$
(H_{\mu_q},f)=(\Theta^{-1}(q),f),
$$
i.e. $H_{\mu_q}=\Theta^{-1}(q)$, as desired.
\end{proofB3}

Thus, so far we have proved that for an arbitrary potential $q\in\mathfrak Q_p$, $p\ge1$, the spectral data $\mathfrak a_q$ of the operator $T_q$ satisfy the conditions $(B_1)$--$(B_4)$. This is the necessity part of Theorem~\ref{Th1}. The next section is devoted to establishing its sufficiency part and solving the inverse spectral problem for the operator $T_q$.

\section{Inverse spectral problem}

In this section we finish the proof of Theorem~\ref{Th1} and prove Theorems~\ref{Th2} and \ref{Th4}, and thus solve the inverse spectral problem for the operator $T_q$. Namely, we show that if a sequence $\mathfrak a:=((\lambda_j,\alpha_j))_{j\in\mathbb Z}$ satisfies the conditions $(B_1)$--$(B_4)$, then there exists a unique potential $q\in\mathfrak Q_p$ such that $\mathfrak a=\mathfrak a_q$. Theorem~\ref{Th4} then suggests a method how to reconstruct this potential from the spectral data.

\subsection{Proof of Theorems~\ref{Th1} and \ref{Th4}}

Here we prove that if a sequence $\mathfrak a$ satisfies the conditions $(B_1)$--$(B_4)$, $\mu:=\mu^{\mathfrak a}$ and $H:=H_\mu$, then $H\in\mathfrak H_p$ and $\mathfrak a=\mathfrak a_q$ for $q=\Theta(H)$. This is the sufficiency part of Theorem~\ref{Th1} and the proof of Theorem~\ref{Th4}.

We start from the following technical but nevertheless important lemma:

\begin{lemma}\label{FHeqLemma}
For an arbitrary sequence $\mathfrak a:=((\lambda_j,\alpha_j))_{j\in\mathbb Z}$ set $\mu:=\mu^{\mathfrak a}$ and $H:=H_\mu$. If $H\in L_p((-1,1),M_r)$, $p\ge1$, then the following equality holds true:
\begin{equation}\label{FHeq}
\mathscr I+\mathscr F_H=\slim\limits_{N\to\infty}
\sum_{j=-N}^N \Phi_0(\lambda_j)\alpha_j\Phi_0^*(\lambda_j),
\end{equation}
where $\mathscr F_H\in\mathscr G_p(M_{2r})$ is an integral operator in $\mathbb H$ with kernel $F_H$ given by (\ref{FH}) and the operators $\Phi_0(\lambda):\mathbb C^r\to\mathbb H$ are given by (\ref{Phi0Def}).
\end{lemma}

\begin{proof}
Given an arbitrary sequence $\mathfrak a:=((\lambda_j,\alpha_j))_{j\in\mathbb Z}$, set $\mu:=\mu^{\mathfrak a}$, $H:=H_\mu$, and assume that $H\in L_p((-1,1),M_r)$. Denote by $\mathscr H$ an integral operator in $L_2((0,1),\mathbb C^r)$ acting by the formula
$$
(\mathscr Hf)(x):=\int\limits_0^1 H(x-t)f(t)\ \d t,\qquad x\in(0,1),
$$
and let $\mathcal I$ stand for the identity operator in $L_2((0,1),\mathbb C^r)$.

Let us show that the operator $\mathcal I+\mathscr H$ is unitarily equivalent to $\mathscr I+\mathscr F_H$. Indeed, consider the unitary transformation $V:L_2((0,1),\mathbb C^r)\to\mathbb H$ given by the formula
$$
(Vf)(t):=\frac1{\sqrt2}\begin{pmatrix}f\left(\frac{1+u}2\right)\\
f\left(\frac{1-u}2\right)\end{pmatrix},\qquad f\in L_2((0,1),\mathbb
C^r).
$$
Then a simple verification shows that
$$
\mathscr I+\mathscr F_H=V(\mathcal I+\mathscr H)V^{-1}.
$$

Furthermore, we find that for all $j\in\mathbb Z$,
$$
\Phi_0(\lambda_j)\alpha_j\Phi_0^*(\lambda_j)
=V\Psi_0(\lambda_j)\alpha_j\Psi_0^*(\lambda_j) V^{-1},
$$
where for an arbitrary $\lambda\in\mathbb C$ the operator $\Psi_0(\lambda)$ acts from $\mathbb C^r$ to $L_2((0,1),\mathbb C^r)$ by the formula
$$
[\Psi_0(\lambda)c](x):=\e^{2\i\lambda x}c,
$$
$\Psi_0^*(\lambda):=[\Psi_0(\lambda)]^*$. Hence, in order to prove (\ref{FHeq}) it suffices to show that
\begin{equation}\label{SmoothAuxEq1}
\mathcal I+\mathscr H=
\slim\limits_{N\to\infty}\sum_{j=-N}^N \Psi_0(\lambda_j)\alpha_j\Psi_0^*(\lambda_j).
\end{equation}

Firstly, we observe that (\ref{SmoothAuxEq1}) holds true on the set of smooth functions $f\in\mathcal S^r$.
Namely, let $f\in\mathcal S^r$, $\supp f\subset[0,1]$. Then, by virtue of definitions (\ref{Ftransform}) and (\ref{AccFirstDef}), we have that for an arbitrary but fixed $x\in(0,1)$,
\begin{align*}
[(\mathcal I+\mathscr H)f](x)=
f(x)+\int\limits_0^1 H(x-t)f(t)\ \d t
=f(x)+\int\limits_{-1}^1 H(s)f(x-s)\ \d s
\\
=\sum_{j=-\infty}^\infty \alpha_j \int\limits_{-1}^1 \e^{2\i\lambda_j s}f(x-s)\ \d s
=\sum_{j=-\infty}^\infty \alpha_j \int\limits_0^1 \e^{2\i\lambda_j (x-t)} f(t)\ \d t,
\end{align*}
noting that the series in the last two expressions are convergent because $f\in\mathcal S^r$.  Thus we conclude that for all $f\in\mathcal S^r$, $\supp f\subset[0,1]$,
\begin{equation}\label{SmoothAuxEq2}
(\mathcal I+\mathscr H)f=
\lim\limits_{N\to\infty}\sum_{j=-N}^N \Psi_0(\lambda_j)\alpha_j\Psi_0^*(\lambda_j)f.
\end{equation}

Since $\{f\in\mathcal S^r\ | \ \supp f\subset[0,1]\}$ is dense everywhere in $L_2((0,1),\mathbb C^r)$, in order to finish the proof it remains to show that there exists the limit $\slim_{N\to\infty}\sum_{j=-N}^N C_j$, where the operators $C_j$, $j\in\mathbb Z$, act in $L_2((0,1),\mathbb C^r)$ by the formula
$$
C_j:=\Psi_0(\lambda_j)\alpha_j\Psi_0^*(\lambda_j).
$$
To this end, note that for every $j\in\mathbb Z$ the operator $C_j$ is self-adjoint and non-negative, and that (\ref{SmoothAuxEq2}) implies that for all $N\in\mathbb N$,
$$
\sum_{j=-N}^N C_j\le\mathcal I+\mathscr H.
$$
Therefore, from well-known theorem on convergence of monotonic sequence of non-negative self-adjoint operators (see, e.g., \cite[Chapter IV, \S2]{Sadovnichy}) we obtain that the sequence $\sum_{j=-N}^N C_j$, $N\in\mathbb N$, is convergent in the strong operator topology. Hence (\ref{SmoothAuxEq1}) follows, and thus the proof is complete.
\end{proof}

Now, the following observation will serve as a part of Theorem~\ref{Th4}:

\begin{lemma}\label{InvAccLemma}
Let a sequence $\mathfrak a$ satisfy the conditions $(B_3)$ and $(B_4)$, set $\mu:=\mu^{\mathfrak a}$ and $H:=H_\mu$. Then $H\in\mathfrak H_p$.
\end{lemma}

\begin{proof}
Firstly, it follows from the condition $(B_4)$ that $H\in L_p((-1,1),M_r)$. Construct the function $F_H\in G_p(M_{2r})$ by formula (\ref{FH}) and let $\mathscr F_H\in\mathscr G_p(M_{2r})$ be an integral operator in $\mathbb H$ with kernel $F_H$. Taking into account Lemma~\ref{FactLemma1}, we find that the present lemma will be proved if we show that $\mathscr I+\mathscr F_H>0$.

It follows from Lemma~\ref{FHeqLemma} that
$$
\mathscr I+\mathscr F_H
=\slim\limits_{N\to\infty} \sum_{j=-N}^N \Phi_0(\lambda_j)\alpha_j\Phi_0^*(\lambda_j).
$$
Therefore, since for every $j\in\mathbb Z$ the operator $\Phi_0(\lambda_j)\alpha_j\Phi_0^*(\lambda_j)$ is non-negative, it follows that $\mathscr I+\mathscr F_H\ge0$. Hence, since the operator $\mathscr F_H$ is compact, we obtain that $\mathscr I+\mathscr F_H>0$ if and only if $\ker(\mathscr I+\mathscr F_H)=\{0\}$.
Now taking into account that for all $j\in\mathbb Z$, $\ker\Phi_0(\lambda_j)=\{0\}$, we find that
$$
\ker(\mathscr I+\mathscr F_H)=
\ker \left(
\sum_{j=-\infty}^\infty \Phi_0(\lambda_j)\alpha_j\Phi_0^*(\lambda_j)
\right)
=\bigcap\limits_{j=-\infty}^\infty \ker \alpha_j\Phi_0^*(\lambda_j)
=U\mathcal X^\bot,
$$
where $\mathcal X$ is the set from the condition $(B_3)$ defined by (\ref{XiB4}) and the unitary mapping $U:L_2((-1,1),M_r)\to\mathbb H$ acts by the formula
$$
(Uf)(x)=\left(\begin{array}{cc}f(x),&f(-x)\end{array}\right)^\top,\qquad x\in(0,1).
$$
Since, by virtue of the condition $(B_3)$, $\mathcal X^\bot=\{0\}$, we obtain that $\ker(\mathscr I+\mathscr F_H)=\{0\}$, as desired.
\end{proof}

Now, given an arbitrary sequence $\mathfrak a$ satisfying the conditions $(B_1)$--$(B_4)$, set $\mu:=\mu^{\mathfrak a}$, $H:=H_\mu$ and
$$
q:=\Theta(H).
$$
For all $j\in\mathbb Z$ define the operators
\begin{equation}\label{hatPdef}
\widetilde P_{\mathfrak a,j}:=\Phi_q(\lambda_j)\alpha_j\Phi_q^*(\lambda_j).
\end{equation}

\begin{remark}
By this construction, if $\mathfrak a=\mathfrak a_q$ for some $q\in\mathfrak Q_p$, then $\widetilde P_{\mathfrak a,j}=P_{q,j}$, i.e. $\widetilde P_{\mathfrak a,j}$ is an orthogonal projector of the operator $T_q$ corresponding to the eigenvalue $\lambda_j(q)$ (see Theorem~\ref{DirPropTh}, parts \it{(iii)} and \it{(iv)}).
\end{remark}

According to (\ref{hatPdef}), for every $j\in\mathbb Z$ the operator $\widetilde P_{\mathfrak a,j}$ is self-adjoint non-negative operator of finite rank. If we show that $\{\widetilde P_{\mathfrak a,j}\}_{j\in\mathbb Z}$ is a complete system of orthogonal projectors, then the same arguments as in \cite{sturm,dirac1} will lead us to the desired equality $\mathfrak a=\mathfrak a_q$, $q=\Theta(H)$.

We start from proving completeness of $\{\widetilde P_{\mathfrak a,j}\}_{j\in\mathbb Z}$, which is a direct consequence of Proposition~\ref{DirectFactTh}:

\begin{lemma}\label{InvAuxLemma1}
Let a sequence $\mathfrak a$ satisfy the conditions $(B_1)$--$(B_4)$, set $\mu:=\mu^{\mathfrak a}$, $H:=H_\mu$ and $q:=\Theta(H)$.
Then
\begin{equation}\label{PjCompl}
\slim\limits_{N\to\infty} \sum_{j=-N}^N \widetilde P_{\mathfrak a,j}=\mathscr I.
\end{equation}
\end{lemma}

\begin{proof}
Let the assumptions of the present lemma hold true. Construct the function $F_H\in G_p(M_{2r})$ by formula (\ref{FH}) and denote by $\mathscr F_H\in\mathscr G_p(M_{2r})$ an integral operator in $\mathbb H$ with kernel $F_H$. It follows from Proposition~\ref{DirectFactTh} that
\begin{equation}\label{PjComplAuxEq1}
\mathscr I+\mathscr F_H=(\mathscr I+\mathscr K_q)^{-1}(\mathscr I+{\mathscr K_q}^*)^{-1},
\qquad q=\Theta(H),
\end{equation}
where $\mathscr K_q\in\mathscr G_p^+(M_{2r})$ is an integral operator in $\mathbb H$ with kernel $K_q$ (see Lemma~\ref{SCLemma}). From the other side, observing that by virtue of Lemma~\ref{FHeqLemma},
\begin{equation}\label{PjComplAuxEq2}
\mathscr I+\mathscr F_H=\slim\limits_{N\to\infty} \sum_{j=-N}^N  \Phi_0(\lambda_j)\alpha_j\Phi_0^*(\lambda_j),
\end{equation}
and recalling that $\Phi_q(\lambda)=(\mathscr I+\mathscr K_q)\Phi_0(\lambda)$ (see (\ref{PhiTransfOp})), by virtue of the equalities (\ref{PjComplAuxEq1}) and (\ref{PjComplAuxEq2}) we arrive at (\ref{PjCompl}).
\end{proof}

Next we prove that $\{\widetilde P_{\mathfrak a,j}\}_{j\in\mathbb Z}$ is a system of pairwise orthogonal projectors. Again, as in solving the direct spectral problem, the proof uses Lemma~\ref{B2MainLemma} and the theory of Riesz bases.

\begin{lemma}\label{InvAuxLemma2}
Let a sequence $\mathfrak a$ satisfy the conditions $(B_1)$--$(B_4)$, set $\mu:=\mu^{\mathfrak a}$, $H:=H_\mu$ and $q:=\Theta(H)$.
Then $\{\widetilde P_{\mathfrak a,j}\}_{j\in\mathbb Z}$ is a system of pairwise orthogonal projectors.
\end{lemma}

\begin{proof}
As in the proof of Proposition~\ref{B2Lemma}, for an arbitrary $j\in\mathbb Z$ we denote by $\{v_{j,k}\}_{k=1}^{\rank\alpha_j}$ a system of pairwise orthogonal vectors in $\mathbb C^r$ such that
$$
\alpha_j=\sum_{k=1}^{\rank\alpha_j} (\ \cdot\ |v_{j,k})v_{j,k}.
$$
Then, taking into account (\ref{hatPdef}), we obtain that
$$
\widetilde P_{\mathfrak a,j}=\sum_{k=1}^{\rank\alpha_j} (\ \cdot\ |\boldsymbol f_{j,k})\boldsymbol f_{j,k},
$$
where $\boldsymbol f_{j,k}:=\Phi_q(\lambda_j)v_{j,k}$.

The present lemma will be proved if we show that the system
$$
\boldsymbol{\mathcal A}:=\{\boldsymbol f_{j,k}\ |\ j\in\mathbb Z,\ \ 1\le k\le\rank\alpha_j\}
$$
is a Riesz basis for the space $\mathbb H$. Indeed, from the equality (\ref{PjCompl}) we then obtain that for an arbitrary function $\boldsymbol f_{s,l}$,
$$
\sum_{j=-\infty}^\infty \sum_{k=1}^{\rank\alpha_j} (\boldsymbol f_{s,l}|\boldsymbol f_{j,k})\boldsymbol f_{j,k}=\boldsymbol f_{s,l},
$$
which implies the equalities $(\boldsymbol f_{j,k}|\boldsymbol f_{j,k})=1$ and $(\boldsymbol f_{s,l}|\boldsymbol f_{j,k})=0$ as $(s,l)\neq(j,k)$. Therefore, $\widetilde P_{\mathfrak a,s}\widetilde P_{\mathfrak a,j}=0$ as $s\neq j$, and thus $\{\widetilde P_{\mathfrak a,j}\}_{j\in\mathbb Z}$ is a system of pairwise orthogonal projectors.

Hence it remains to prove that $\boldsymbol{\mathcal A}$ is a Riesz basis for $\mathbb H$. To this end, it suffices to show that the system
\begin{equation}\label{InvAsystem}
\mathcal B_0:=\left\{\e^{\i\lambda_jt}v_{j,k}\ \big|\ j\in\mathbb Z,\ \ 1\le k\le\rank\alpha_j \right\}
\end{equation}
is a Riesz basis for the space $\mathcal H:=L_2((-1,1),\mathbb C^r)$. Indeed, introduce the unitary mapping $U:\mathcal H\to\mathbb H$ acting by the formula
$$
(Uf)(t)=\begin{pmatrix}f(-t),&
f(t)\end{pmatrix}^\top,\qquad t\in(0,1),
$$
and note that $U$ maps $\e^{\i\lambda_j x}v_{j,k}$ to $\Phi_0(\lambda_j)v_{j,k}$ for all $j\in\mathbb Z$, $1\le k\le\rank\alpha_j$. Therefore, if $\mathcal B_0$ is a Riesz basis for $\mathcal H$, then the system
$$
\boldsymbol{\mathcal A}_0:=\{\Phi_0(\lambda_j)v_{j,k}\ |\ j\in\mathbb Z,\
\ 1\le k\le\rank\alpha_j\}
$$
is a Riesz basis for $\mathbb H$. Then, since $\Phi_q(\lambda)=(\mathscr I+\mathscr K_q)\Phi_0(\lambda)$ and $\mathscr I+\mathscr K_q$ is a homeomorphism of the space $\mathbb H$, we obtain that $\boldsymbol{\mathcal A}$ remains a Riesz basis for $\mathbb H$, as desired.

Thus let us show that $\mathcal B_0$ is a Riesz basis for $\mathcal H$. We do this by applying Lemma~\ref{B2MainLemma} and Theorem~\ref{QuadrCloseLemma}. Firstly, observe that by virtue of the conditions $(B_1)$ and $(B_2)$, there exists $N_0\in\mathbb N$ such that for all $n\in\mathbb Z$, $|n|>N_0$,
\begin{equation}\label{InvBaseEq1}
\sum_{\lambda_j\in\Delta_n}\rank\alpha_j=r.
\end{equation}
Indeed, it follows from the condition $(B_1)$ that there exists $N_0\in\mathbb N$ such that
$$
\bigg\|I-\sum\limits_{\lambda_j\in\Delta_n}\alpha_j\bigg\|<1,\qquad |n|>N_0.
$$
Hence, $\sum_{\lambda_j\in\Delta_n}\rank \alpha_j\ge r$ as $|n|>N_0$. Moreover, by virtue of the condition $(B_2)$, $N_0$ can be taken so large that
$$
\sum\limits_{n=-N}^N\sum\limits_{\lambda_j\in\Delta_n}\rank \alpha_j=(2N+1)r
$$
for all natural $N\ge N_0$. Therefore,
$
\sum_{\lambda_j\in\Delta_n}\rank \alpha_j+\sum_{\lambda_j\in\Delta_{-n}}\rank \alpha_j=2r
$
as $|n|>N_0$, and thus we arrive at (\ref{InvBaseEq1}).

Now, since a sequence $\mathfrak a$ satisfies the condition $(B_1)$ and for all $n\in\mathbb Z$, $|n|>N_0$, the equality (\ref{InvBaseEq1}) holds true, we find that the conditions of Lemma~\ref{B2MainLemma} are satisfied. Therefore, from Lemma~\ref{B2MainLemma} we obtain that there exists a large enough natural $N>N_0$ such that the system $\mathcal E_N\cup\mathcal B_N$, where
$$
\mathcal E_N:=\left\{\frac1{\sqrt2}\e^{\i\pi nt}\epsilon_s\ \big|\ n\in\mathbb Z,\ |n|\le N,
\quad s=1,\ldots,r \right\}
$$
and
$$
\mathcal B_N:=\left\{\e^{\i\lambda_jt}v_{j,k}\ \big|\ \lambda_j\in\Delta_n,\ |n|>N,
\quad 1\le k\le\rank\alpha_j \right\},
$$
is a Riesz basis for the space $\mathcal H$.

Finally, observe that by virtue of the condition $(B_3)$, the system $\mathcal B_0$ is complete in $\mathcal H$. Moreover, since $\mathcal B_N\subset\mathcal B_0$ and, by virtue of the condition $(B_2)$, the finite systems
$$
\mathcal B_0\setminus\mathcal B_N=\left\{\e^{\i\lambda_jt}v_{j,k}\ \big|\ \lambda_j\in\Delta_n,\ |n|\le N,
\quad 1\le k\le\rank\alpha_j \right\}
$$
and $\mathcal E_N$ consist of the same number of elements, we find that $\mathcal B_0$ is quadratically close to $\mathcal E_N\cup\mathcal B_N$. Then it follows from Theorem~\ref{QuadrCloseLemma} that $\mathcal B_0$ remains a Riesz basis for $\mathcal H$, as desired.
\end{proof}

Now we are ready to prove Theorem~\ref{Th4}:

\begin{proofTh4}
Firstly, it follows from Lemma~\ref{InvAccLemma} that if a sequence $\mathfrak a$ satisfies the conditions $(B_1)$--$(B_4)$, $\mu:=\mu^{\mathfrak a}$ and $H:=H_\mu$, then $H\in\mathfrak H_p$. Thus, it only remains to prove that $\mathfrak a=\mathfrak a_q$ for $q=\Theta(H)$.

The proof of this claim repeats the technique that was suggested in \cite{sturm}. Namely, as in \cite{sturm,dirac1}, we observe that it is enough to prove the embedding
\begin{equation}\label{PrRel}
\Ran\widetilde P_{\mathfrak a,j}\subset\ker(T_q-\lambda_j\mathscr I),\qquad j\in\mathbb Z,
\end{equation}
where the operators $\widetilde P_{\mathfrak a,j}$ are given by (\ref{hatPdef}). Indeed, taking into account completeness of $\{\widetilde P_{\mathfrak a,j}\}_{j\in\mathbb Z}$, from (\ref{PrRel}) we immediately conclude that $\lambda_j=\lambda_j(q)$ for all $j\in\mathbb Z$, where $\lambda_j(q)$ are eigenvalues of $T_q$. Now, from this equality and from (\ref{PrRel}) we obtain that $P_{q,j}-\widetilde P_{\mathfrak a,j}\ge0$, $j\in\mathbb Z$, where $P_{q,j}$ are the orthogonal projectors of the operator $T_q$ (see Theorem~\ref{DirPropTh}). However, taking into account completeness of the systems $\{\widetilde P_{\mathfrak a,j}\}_{j\in\mathbb Z}$ and $\{P_{q,j}\}_{j\in\mathbb Z}$, we observe that $\sum_{j=-\infty}^\infty(P_{q,j}-\widetilde P_{\mathfrak a,j})=0$, and thus $P_{q,j}-\widetilde P_{\mathfrak a,j}=0$ for all $j\in\mathbb Z$. Therefore, recalling the representation (\ref{ProjForm}) for $P_{q,j}$, we find that
$$
\Phi_q(\lambda_j)\{\alpha_j(q)-\alpha_j\}\Phi_q^*(\lambda_j)=0,
\qquad j\in\mathbb Z,
$$
and taking into account (\ref{kerPhi}) we arrive at $\alpha_j=\alpha_j(q)$. Together with $\lambda_j=\lambda_j(q)$ this means that $\mathfrak a=\mathfrak a_{q}$, as desired.

Now let us prove (\ref{PrRel}). Firstly, from the definition (\ref{hatPdef}) of $\widetilde P_{\mathfrak a,j}$ we obtain that
$$
\Ran \widetilde P_{\mathfrak a,j}=\{\varphi_{q}(\cdot,\lambda_j)\alpha_jc\ | \ c\in\mathbb C^r\}.
$$
From the other side, it follows from Lemma~\ref{kerTLemma} that
$$
\ker (T_q-\lambda_j \mathscr I)=\{ \varphi_q(\cdot,\lambda_j)c \ |
\ a\varphi_q(1,\lambda_j)c=0 \}.
$$
Therefore, it suffices to show that $a\varphi_q(1,\lambda_j)\alpha_j=0$. The proof of this claim is technical and literally repeats the proof in \cite{dirac1}.
\end{proofTh4}

Finally, the proof of Theorem~\ref{Th1} is straightforward:

\begin{proofTh1}
Firstly, by virtue of Propositions~\ref{B1Lemma}, \ref{B2Lemma} and Corollaries~\ref{B3Lemma}, \ref{B4Lemma} we have that for all $q\in\mathfrak Q_p$, $p\ge1$, the spectral data $\mathfrak a_q$ of the operator $T_q$ satisfy the conditions $(B_1)$--$(B_4)$. This is the necessity part of Theorem~\ref{Th1}. Conversely, if a sequence $\mathfrak a$ satisfies the conditions $(B_1)$--$(B_4)$, $\mu:=\mu^{\mathfrak a}$ and $H:=H_\mu$, then it follows from Theorem~\ref{Th4} that $H\in\mathfrak H_p$ and $\mathfrak a=\mathfrak a_q$ for $q=\Theta(H)$. This is the sufficiency part of the theorem.
\end{proofTh1}

\subsection{Proof of Theorem~\ref{Th2}}

Thus it only remains to prove Theorem~\ref{Th2} stating that the spectral data of the operator $T_q$ determine the potential $q$ uniquely. The proof of this claim repeats the proof in \cite{dirac1}:

\begin{proofTh2}
Evidently, the theorem will be proved if we prove the implication $\mathfrak a_{q_1}=\mathfrak a_{q_2} \Rightarrow q_1=q_2$.
Therefore, let $q_1,q_2\in\mathfrak Q_p$, $p\ge1$, and assume that $\mathfrak a_{q_1}=\mathfrak a_{q_2}=:\mathfrak a$. Let us show that $q_1=q_2$.
Set $H:=H_\mu$, $\mu:=\mu^{\mathfrak a}$, and construct the function $F_H\in G_p(M_{2r})$ by the formula (\ref{FH}). Denote by $\mathscr F_H\in\mathscr G_p(M_{2r})$ an integral operator in $\mathbb H$ with kernel $F_H$. Since
$$
\mathscr I+\mathscr F_H=\slim\limits_{N\to\infty}\sum_{j=-N}^N \Phi_0(\lambda_j)\alpha_j\Phi_0^*(\lambda_j),
$$
by virtue of the equalities (\ref{PhiTransfOp}), (\ref{ResIdent}) and (\ref{ProjForm}) we find that
$$
\mathscr I+\mathscr F_H=
(\mathscr I+\mathscr K_{q_1})^{-1}(\mathscr I+{\mathscr K_{q_1}}^*)^{-1}
=(\mathscr I+\mathscr K_{q_2})^{-1}(\mathscr I+{\mathscr K_{q_2}}^*)^{-1},
$$
where $\mathscr K_{q_j}$ is an integral operator in $\mathbb H$ with kernel $K_{q_j}$, $j=1,2$. Since the operator $\mathscr I+\mathscr F_H$ may admit at most one factorization in $\mathscr G_p(M_{2r})$, we arrive at the equality
$$
\mathscr K_{q_1}=\mathscr K_{q_2}.
$$
Thus it remains to prove the implication $\mathscr K_{q_1}=\mathscr K_{q_2} \Rightarrow q_1=q_2$. By virtue of (\ref{varphiCauchyPr}) and (\ref{VarphiTrasfOp}), this can be obtained from the uniqueness theorem repeating the proof in \cite{dirac1}.
\end{proofTh2}

\subsection*{Acknowledgment}
The author is extremely grateful to his supervisor Ass. Prof. Yaroslav~Mykytyuk for countless discussions, valuable ideas and permanent attention to this work. The author also would like to thank Dr. Rostyslav~Hryniv for valuable suggestions in preparing this manuscript.

\appendix

\setcounter{section}{0}
\def\thesection{\Alph{section}}

\section{Spaces}\label{AppSpaces}

Here we introduce several spaces that are used in this paper.

For an arbitrary $p\ge1$ we denote by $G_p(M_r)$ the set of all measurable functions $K:[0,1]^2\to M_r$ such that for all $x$ and $t$ in $[0,1]$ the functions $K(x,\cdot)$ and $K(\cdot,t)$ belong to $L_p((0,1),M_r)$ and, moreover, the mappings
$$
[0,1]\ni x\mapsto K(x,\cdot)\in L_p((0,1),M_r), \qquad
[0,1]\ni t\mapsto K(\cdot,t)\in L_p((0,1),M_r)
$$
are continuous. The set $G_p(M_r)$ becomes a Banach space upon introducing the norm
$$
\|K\|_{G_p}=\max\left\{
\max\limits_{x\in[0,1]}\|K(x,\cdot)\|_{L_p},\
\max\limits_{t\in[0,1]}\|K(\cdot,t)\|_{L_p}\right\}.
$$

We denote by $\mathscr G_p(M_r)$ the set of all integral operators in $L_2((0,1),\mathbb C^r)$ with kernels $K\in G_p(M_r)$ and endow $\mathscr G_p(M_r)$ with the norm
$$
\|\mathscr K\|_{\mathscr G_p}:=\|K\|_{G_p},\qquad
\mathscr K\in\mathscr G_p(M_r).
$$
The space $\mathscr G_p(M_r)$ forms a subalgebra in the algebra $\mathscr B_\infty$ of compact operators in $L_2((0,1),\mathbb C^r)$.

We set
$$
\Omega:=\{(x,t) \ | \ 0\le t\le x\le 1\}, \qquad \Omega^-:=[0,1]^2\setminus\Omega,
$$
and write $G_p^+(M_r)$ for the set of all functions $K\in G_p(M_r)$ such that $K(x,t)=0$ a.e. in $\Omega^-$, and $G_p^-(M_r)$ for set of all $K\in G_p(M_r)$ such that $K(x,t)=0$ a.e. in $\Omega$. By $\mathscr G_p^\pm(M_r)$ we denote the subalgebras of $\mathscr G_p(M_r)$ consisting of all operators $\mathscr K\in\mathscr G_p(M_r)$ with kernels $K\in G_p^\pm(M_r)$.

Let $\mathscr I$ stand for the identity operator in $L_2((0,1),\mathbb C^r)$. Then the following lemma is established in \cite{sturm} for the case $p=2$. However, its generalization to the case of an arbitrary $p\ge1$ is straightforward:

\begin{lemma}\label{KinvLemma}
For all $p\ge1$, the mapping $\mathscr K\mapsto (\mathscr I+\mathscr K)^{-1}-\mathscr I$ acts continuously in $\mathscr G_p^+(M_r)$.
\end{lemma}

We denote by $\mathcal P$ the set of all polynomials over complex numbers and set
$$
\mathcal P^r:=\{(f_1,\ldots,f_r)^\top\ |\ f_j\in\mathcal P,\ j=1,\ldots,r\},\qquad r\in\mathbb N.
$$
We denote by $\mathcal S$ the Schwartz space of smooth functions $f\in C^\infty(\mathbb R)$ whose derivatives (including the function itself) decay at infinity faster than any power of $|x|^{-1}$, i.e.
$$
\mathcal S:=\{f\in C^\infty(\mathbb R)\ |\ x^\alpha \mathrm D^\beta f(x)\to0\ as\ |x|\to\infty,\quad
\alpha,\beta\in\mathbb N\cup\{0\}\}.
$$
Similarly, we set
$$
\mathcal S^r:=\{(f_1,\ldots,f_r)^\top\ |\ f_j\in\mathcal S,\ j=1,\ldots,r\},\qquad r\in\mathbb N.
$$

Also, we formulate here the following refined version of the Riemann--Lebesgue lemma that is mentioned, e.g., in \cite{trushzeros}:

\begin{lemma}\label{RefRiemannLebesgue}
Let $g\in L_1((-1,1),M_r)$. Then
$$
\lim\limits_{\mathbb C\owns\lambda\to\infty} \ \e^{-|\Im\lambda|}\int\limits_{-1}^1 \e^{\i\lambda t}g(t)\ \d t=0
$$
in the metric of the space $M_r$.
\end{lemma}

\section{Factorization of integral operators}\label{AppFact}

In this appendix we recall some facts from the theory of factorization of integral operators (see \cite{MykFact1,MykFact2}), which are also mentioned in \cite{MykDirac,sturm,dirac1}. See also \cite{kreinvolterra} for the details.

We say that an operator $\mathscr I+\mathscr F$, where $\mathscr F\in \mathscr G_p(M_r)$, $p\ge1$, admits a factorization in $\mathscr G_p(M_r)$ if there exist the operators $\mathscr L^+\in \mathscr G_p^+(M_r)$ and $\mathscr L^-\in \mathscr G_p^-(M_r)$ such that
\begin{equation}\label{AppFactEq1}
\mathscr I+\mathscr F=(\mathscr I+\mathscr L^+)^{-1}(\mathscr I+\mathscr L^-)^{-1}.
\end{equation}
If $\mathscr F$ is self-adjoint, then $\mathscr L^-=(\mathscr L^+)^*$. This follows from uniqueness of $\mathscr L^\pm$ (see Theorem~\ref{FactTh1} below).

The following two theorems are established in \cite{MykFact1,MykFact2} for the case $p=2$. Their generalization for the case of an arbitrary $p\ge1$ is mentioned in \cite{MykDirac}.

\begin{theorem}\label{FactTh1}
If $\mathscr I+\mathscr F$, $\mathscr F\in \mathscr G_p(M_r)$, admits a factorization in $\mathscr G_p(M_r)$, then the corresponding operators $\mathscr L^\pm=\mathscr L^\pm(\mathscr F)$ in the representation (\ref{AppFactEq1}) are unique. Moreover, the set of operators $\mathscr F\in\mathscr G_p(M_r)$, such that $\mathscr I+\mathscr F$ admits a factorization, is open in $\mathscr G_p(M_r)$ and the mappings
$$
\mathscr G_p(M_r)\owns\mathscr F\mapsto\mathscr L^\pm(\mathscr F)\in\mathscr G_p(M_r)
$$
are continuous.
\end{theorem}

\begin{theorem}\label{FactTh2}
An operator $\mathscr I+\mathscr F$, $\mathscr F\in \mathscr G_p(M_r)$, admits a factorization in $\mathscr G_p(M_r)$ if and only if: (A) the operators $\mathscr I+\chi_a\mathscr F\chi_a$ have trivial kernels for all $a\in[0,1]$, where $\chi_a$ is an operator of multiplication by the indicator of the interval $[0,a]$,
i.e.
$$
(\chi_af)(x):=\left\{\begin{array}{cl} f(x), & x\in[0,a], \\ 0, & x\in(a,1]. \end{array}\right.
$$
If an operator $\mathscr F$ is self-adjoint, then the condition (A) is equivalent to positivity of $\mathscr I+\mathscr F$.
\end{theorem}

Now we are interested in a connection between Krein's accelerants (see Definition~\ref{Def2}) and factorization of some integral operators.

Let $H\in L_p((-1,1),M_r)$, $p\ge1$, be an arbitrary function such that $H(-x)=H(x)^*$ for almost all $x\in(-1,1)$.
Denote by $\mathscr H$ an integral operator in $L_2((0,1),\mathbb C^r)$ acting by the formula
$$
(\mathscr Hf)(x):=\int\limits_0^1 H(x-t)f(t)\ \d t,\qquad x\in(0,1).
$$
Now set
$$
\mathbb H:=L_2((0,1),\mathbb C^r)\times L_2((0,1),\mathbb C^r).
$$
Let $\mathcal I$ stand for the identity operator in $L_2((0,1),\mathbb C^r)$ and $\mathscr I$ for the identity operator in $\mathbb H$. Then the following lemma holds true:

\begin{lemma}\label{FactLemma1}
Let $H\in L_p((-1,1),M_r)$, $H(-x)=H(x)^*$ for almost all $x\in(-1,1)$, and let $\mathscr F_H\in\mathscr G_p(M_{2r})$ be an integral operator in $\mathbb H$ with kernel $F_H\in G_p(M_{2r})$ taking the form
$$
F_H(x,t)=\frac{1}{2}\begin{pmatrix}H\left(\frac{x-t}{2}\right)&
H\left(\frac{x+t}{2}\right)\\
H\left(-\frac{x+t}{2}\right)&
H\left(-\frac{x-t}{2}\right)\end{pmatrix},\qquad 0\le x,t\le1.
$$
Then the following statements are equivalent:
\begin{itemize}
\item[(i)]the operator $\mathcal I+\mathscr H$ admits a factorization in $\mathscr G_p(M_r)$;
\item[(ii)]the operator $\mathscr I+\mathscr F_H$ admits a factorization in $\mathscr G_p(M_{2r})$;
\item[(iii)]the function $H$ is an accelerant, i.e. $H\in\mathfrak H_p$;
\item[(iv)]the Krein equation
\begin{equation}\label{AppKreinEq}
R(x,t)+H(x-t)+\int\limits_0^x R(x,s)H(x-s)\ \d s=0,\qquad (x,t)\in\Omega,
\end{equation}
is solvable in $G_p^+(M_r)$.
\end{itemize}
\end{lemma}

\begin{proof}
Let us establish the equivalence of {\it (i)} and {\it (ii)}. Firstly, observe that both operators $\mathcal I+\mathscr H$ and $\mathscr I+\mathscr F_H$ are self-adjoint, and therefore both of them admit a factorization if and only if they are positive. Now consider the unitary transformation $V:L_2((0,1),\mathbb C^r)\to\mathbb H$ given by the formula
$$
(Vf)(t):=\frac1{\sqrt2}\begin{pmatrix}f\left(\frac{1+u}2\right)\\
f\left(\frac{1-u}2\right)\end{pmatrix},\qquad f\in L_2((0,1),\mathbb C^r),
$$
and verify that
$$
\mathscr I+\mathscr F_H=V(\mathcal I+\mathscr H)V^{-1}.
$$
Therefore, the operators $\mathcal I+\mathscr H$ and $\mathscr I+\mathscr F_H$ are unitarily equivalent, and hence $\mathcal I+\mathscr H>0$ if and only if $\mathscr I+\mathscr F_H>0$. Thus the equivalence of {\it (i)} and {\it (ii)} follows.

The equivalence of {\it (ii)} and {\it (iii)} obviously follows from Theorem~\ref{FactTh2} and Definition~\ref{Def2}. Finally, the equivalence of {\it (ii)} and {\it (iv)} is proved in \cite{MykDirac}.
\end{proof}

For an arbitrary accelerant $H\in\mathfrak H_p$, a solution of the Krein equation (\ref{AppKreinEq}) is unique, and we denote it by $R_H$. The mapping $\mathfrak H_p\owns H\mapsto R_H\in G_p^+(M_r)$ is continuous (see \cite{MykDirac}).

\section{Riesz bases}\label{AppRiesz}

Here we state some facts from the theory of Riesz bases.

\begin{definition}
We say that two bases $(x_n)_{n\in\mathbb N}$ and $(y_n)_{n\in\mathbb N}$ in a Banach space $X$ are equivalent if there exists a bounded and boundedly invertible operator $T:X\to X$ such that $Tx_n=y_n$ for all $n\in\mathbb N$.
A basis for a Hilbert space $\mathcal H$ is called a Riesz basis if it is equivalent to some orthonormal basis for $\mathcal H$.
\end{definition}

We set
$$
\mathcal H:=L_2((-1,1),\mathbb C^r)
$$
and endow $\mathcal H$ with the inner product
$$
(f|g)_{\mathcal H}:=\int\limits_{-1}^1 (f(t)|g(t))_{\mathbb C^r}\ \d t,
\qquad f,g\in\mathcal H.
$$
The set $\mathcal H$ with the inner product $(\cdot|\cdot)_{\mathcal H}$ becomes a Hilbert space. Throughout this appendix we denote by $\widehat f$ the Fourier transform in $\mathcal H$ given by
$$
\widehat f(n):=\frac1{\sqrt2}\int\limits_{-1}^1 \e^{-\i\pi nt}f(t)\ \d t,
\qquad n\in\mathbb Z,\ \ f\in\mathcal H.
$$

We start from a certain analogue of well-known Kadec's $1/4$-theorem (see, e.g., \cite[Chapter 1]{young}).
Let $\xi:=(\xi_n)_{n\in\mathbb Z}$ be a non-decreasing sequence of real numbers such that
$$
\lim\limits_{n\to+\infty}\xi_n=+\infty,\qquad \lim\limits_{n\to-\infty}\xi_n=-\infty,
$$
and let $v:=(v_n)_{n\in\mathbb Z}$ be a sequence of non-zero vectors in $\mathbb C^r$.
Consider the system of functions
$$
\mathscr E(\xi,v):=\{\e^{\i\xi_nt}v_n\ |\ n\in\mathbb Z\}
$$
in the space $\mathcal H$. We are interested in finding conditions guaranteeing  that the system $\mathscr E:=\mathscr E(\xi,v)$ forms a Riesz basis for the space $\mathcal H$.

\begin{definition}
We say that the system $\mathscr E$ enjoys the condition $(R_0)$ if for all $n\in\mathbb Z$ the numbers $\xi_{n,k}:=\xi_{nr+k}$, $k=1,\ldots,r$, belong to the interval
$$
\Delta_n:=\left(\pi n-\frac{\pi}{2},\pi n+\frac{\pi}{2}\right]
$$
and the vectors $v_{n,k}:=v_{nr+k}$, $k=1,\ldots,r$, form a basis for $\mathbb C^r$.
\end{definition}

With an arbitrary system $\mathscr E$ enjoying the condition $(R_0)$ we associate two sequences $(A_n)_{n\in\mathbb Z}$ and $(B_n)_{n\in\mathbb Z}$ of bounded linear operators in $\mathbb C^r$ acting by the formulae
$$
A_n v_{n,k}:=\xi_{n,k}v_{n,k},\qquad B_n \epsilon_k:=v_{n,k},\qquad k=1,\ldots,r,
$$
where $(\epsilon_k)_{k=1}^r$ is a standard orthonormal basis for $\mathbb C^r$.

\begin{definition}
We say that the system $\mathscr E$ enjoys the condition $(R_1)$ if it enjoys the condition $(R_0)$ and, moreover,
$$
\sup\limits_{n\in\mathbb Z}\|A_n-\pi nI\|<\ln2
$$
and
\begin{equation}\label{KadecVcond}
\sup\limits_{n\in\mathbb Z}(\|B_n\|+\|{B_n}^{-1}\|)<\infty.
\end{equation}
\end{definition}

\begin{theorem}\label{KadecTh}
If the system $\mathscr E$ enjoys the condition $(R_1)$, then it forms a Riesz basis for the space $\mathcal H$.
\end{theorem}

\begin{proof}
Denote by $\mathscr E_0$ an orthonormal basis in the space $\mathcal H$ given by
$$
\mathscr E_0:=\left\{\frac1{\sqrt2}\e^{\i\pi nt}\epsilon_k\ |\ n\in\mathbb Z,\ k=1,\ldots,r\right\}.
$$
The present theorem will be proved if we show that there exists a bounded and boundedly invertible operator $S:\mathcal H\to\mathcal H$ that maps $\frac1{\sqrt2}\e^{\i\pi nt}\epsilon_k$ to $\e^{\i\xi_{n,k}t}v_{n,k}$ for all $n\in\mathbb Z$, $k=1,\ldots,r$.

To this end, consider the operators $S$ and $S_0$ that are defined on the linear span $\lin\mathscr E_0$ by the formulae
$$
(Sf)(t)=\sum_{n=-\infty}^\infty \e^{\i A_n t}B_n\widehat f(n),\qquad
(S_0f)(t)=\frac1{\sqrt2}\sum_{n=-\infty}^\infty \e^{\i A_n t}\widehat f(n),
$$
$f\in\lin\mathscr E_0$. Observe that
$$
S\left(\frac1{\sqrt2}\e^{\i\pi nt}\epsilon_k\right)=\e^{\i A_n t}B_n\epsilon_k
=\e^{\i A_n t}v_{n,k}=\e^{\i\xi_{n,k}t}v_{n,k}
$$
for all $n\in\mathbb Z$, $k=1,\ldots,r$, and that the equality $S=S_0B$ takes place, where $B:\mathcal H\to\mathcal H$ acts by the formula
$$
(\widehat{Bf})(n)=\sqrt2 B_n\widehat f(n),\qquad n\in\mathbb Z.
$$
Since continuity and invertibility of the operator $B$ follows from the condition (\ref{KadecVcond}), it suffices to show that the operator
$$
\mathcal B f:=S_0f-f,\qquad f\in\lin\mathscr E_0,
$$
has the norm less than 1.

Set $\widetilde A_n:=A_n-\pi nI$, $n\in\mathbb Z$. Since
$$
\e^{\i A_nt}-\e^{\i\pi ntI}=\e^{\i\pi nt}(\e^{\i\widetilde A_nt}-I)=
\e^{\i\pi nt}\sum_{k=1}^{\infty}\frac{(\i\widetilde A_nt)^k}{k!}
$$
and
$$
f(t)=\frac1{\sqrt2}\sum_{n=-\infty}^\infty \e^{\i\pi nt}\widehat f(n),
$$
we find that
\begin{align*}
\sqrt2\|S_0f-f\|_{\mathcal H}=\left\| \sum_{n=-\infty}^\infty
\e^{\i\pi nt}(\e^{\i\widetilde A_nt}-I)\widehat f(n) \right\|_{\mathcal H}
\\
=
\left\| \sum_{n=-\infty}^\infty \sum_{k=1}^\infty
\e^{\i\pi nt}\frac{(\i \widetilde A_n t)^k}{k!}\widehat f(n) \right\|_{\mathcal H}
\le
\sum_{k=1}^\infty \frac{1}{k!} \left\| \sum_{n=-\infty}^\infty \e^{\i\pi nt}\widetilde A_n^k\widehat f(n)
\right\|_{\mathcal H}
\\
\le
\sum_{k=1}^\infty \frac{\sqrt2}{k!} \left( \sum_{n=-\infty}^\infty
\|\widetilde A_n^k\widehat f(n)\|^2 \right)^{1/2}
\le
\sqrt2\sum_{k=1}^\infty \frac{\delta^k}{k!} \|f\|_{\mathcal H},
\end{align*}
where $\delta:=\sup_{n\in\mathbb Z}\|\widetilde A_n\|<\ln2$. Therefore, $\|S_0f-f\|\le(\e^\delta-1)\|f\|$, i.e. $\|S_0-\mathscr I\|\le \e^\delta-1<1$, as desired.
\end{proof}


\begin{definition}
Two sequences of vectors $(f_n)_{n\in\mathbb Z}$ and $(g_n)_{n\in\mathbb Z}$ in a Hilbert space are said to be quadratically close if
$$
\sum_{n=-\infty}^\infty \|f_n-g_n\|^2<\infty.
$$
\end{definition}
The following theorem is a simple consequence of Theorem~15 in \cite[Chapter 1]{young}:

\begin{theorem}\label{QuadrCloseLemma}
Let $(f_n)_{n\in\mathbb Z}$ be a Riesz basis for $\mathcal H$. If $(g_n)_{n\in\mathbb Z}$ is complete and quadratically close to $(f_n)_{n\in\mathbb Z}$, then $(g_n)_{n\in\mathbb Z}$ is a Riesz basis for $\mathcal H$.
\end{theorem}


\begin{thebibliography}{99}



\bibitem{LevSargs1}B.~M.~Levitan and I.~S.~Sargsjan, \textit{Sturm--Liouville and Dirac operators} Kluwer Academic Publishers, 1991.

\bibitem{Marchenko}V.~A.~Marchenko, \textit{Sturm--Liouville operators and applications} Birkh\"auser, 1986.

\bibitem{Thaller}B.~Thaller, \textit{The Dirac equation} Springer, 1992.

\bibitem{GasLev1}M.~G.~Gasymov and B.~M.~Levitan, \textit{The inverse problem for the {D}irac system} Dokl. Akad. Nauk SSSR \textbf{167} (1966), 967--70 (in Russian).

\bibitem{GasLev2}M.~G.~Gasymov and B.~M.~Levitan, \textit{Determination of the Dirac system from the scattering phase} Dokl. Akad. Nauk SSSR \textbf{167} (1966), 1219--22 (in Russian).

\bibitem{Malamud1}M.~M.~Malamud, \textit{Borg-type theorems for first-order systems on a finite interval} Funct. Anal. Appl. \textbf{33} (1999), 64--8.

\bibitem{Malamud2}M.~M.~Malamud, \textit{Uniqueness questions in inverse problems for systems of differential equations on a finite interval} Trans. Moscow Math. Soc. \textbf{60} (1999), 173--224.

\bibitem{Malamud3}M.~Lesch and M.~Malamud, \textit{The inverse spectral problem for first order systems on the half line} Oper. Theory Adv. Appl. \textbf{117} (2000), 199--238.

\bibitem{ClarkGesz}S.~Clark and F.~Gesztesy, \textit{Weyl--{T}itchmarsh {M}-function asymptotics, local uniqueness results, trace formulas, and {B}org-type theorems for {D}irac operators} Trans. Amer. Math. Soc. \textbf{354} (2002), 3475--534.

\bibitem{KisMakGesz}F.~Gesztesy, A.~Kiselev and K.~A.~Makarov, \textit{Uniqueness results for matrix-valued {S}chr{\"o}dinger, {J}acobi, and {D}irac-type operators} Math. Nachr. \textbf{239/240} (2002), 103--45.

\bibitem{GeszOpV}F.~Gesztesy, R.~Weikard and M.~Zinchenko, \textit{Initial value problems and Weyl--Titchmarsh theory for Schr\"odinger operators with operator-valued potentials} {\it Preprint} arXiv:1109.1613.

\bibitem{Sakh1}A.~Sakhnovich, \textit{Dirac type and canonical systems: spectral and Weyl--Titchmarsh matrix functions, direct and inverse problems} Inverse Problems \textbf{18} (2002), 331--448.

\bibitem{Sakh2}A.~Sakhnovich, \textit{Dirac type system on the axis: explicit formulae for matrix potentials with singularities and soliton-positon interactions} Inverse Problems \textbf{19} (2003), 845--54.

\bibitem{Korot1}D.~Chelkak and E.~Korotyaev, \textit{Parametrization of the isospectral set for the vector-valued {S}turm--{L}iouville problem} J. Funct. Anal. \textbf{241} (2006), 359--73.

\bibitem{Korot2}D.~Chelkak and E.~Korotyaev, \textit{Weyl--{T}itchmarsh functions of vector-valued {S}turm--{L}iouville operators on the unit interval} J. Funct. Anal. \textbf{257} (2009), 1546--88.

\bibitem{MykDirac}S.~Albeverio, R.~Hryniv and Ya.~Mykytyuk, \textit{Inverse spectral problems for {D}irac operators with summable potentials} Russ. J. Math. Phys. \textbf{12} (2005), 406--23.

\bibitem{sturm}Ya.~V.~Mykytyuk and N.~S.~Trush, \textit{Inverse spectral problems for {S}turm{--}{L}iouville operators with matrix-valued potentials} Inverse Problems \textbf{26} (2010), 015009.

\bibitem{dirac1}Ya.~V.~Mykytyuk and D.~V.~Puyda, \textit{Inverse spectral problems for Dirac operators on a finite interval} J. Math. Anal. Appl. \textbf{386} (2012), 177--94.

\bibitem{trushzeros}N.~Trush, \textit{Asymptotics of singular values of entire matrix-valued sine-type functions} Mat. Stud. \textbf{30} (2008), 95--7.

\bibitem{myktrushzeros}Ya.~V.~Mykytyuk and N.~S.~Trush, \textit{Asymptotics of zeros for some entire sine-type functions acting in a Banach algebra} Math. Bull. Shevchenko Sci. Soc. \textbf{4} (2007), 214--9.

\bibitem{mykhrynzeros}R.~O.~Hryniv and Ya.~V.~Mykytyuk, \textit{On zeros of some entire functions} Trans. Amer. Math. Soc. \textbf{361} (2009), 2207--23.

\bibitem{kuchment}P.~Kuchment, \textit{Quantum graphs: an introduction and a brief survey} Analysis on graphs and its applications, Proc. Sympos. Pure Math. \textbf{77} (2008), 291--312.

\bibitem{cauchy}N.~Trush, \textit{Solutions of the {C}auchy problem for factorized {S}turm--{L}iouville equation in a {B}anach algebra} Mat. Stud. \textbf{31} (2008), 75--82.

\bibitem{5auth}D.~Alpay, I.~Gohberg, M.~A.~Kaashoek, L.~Lerer and A.~L.~Sakhnovich, \textit{Krein systems and canonical systems on a finite interval: Accelerants with a jump discontinuity at the origin and continuous potentials} Integr. Equ. Oper. Theory \textbf{68(1)} (2010), 115--50.

\bibitem{geszHerg}F.~Gesztesy and S.~Tsekanovskii, \textit{On matrix-valued Herglotz functions} Math. Nachr. \textbf{218} (2000), 61--138.

\bibitem{Sadovnichy}V.~A.~Sadovnichiy, \textit{Theory of operators} Consultants Bureau, 1991.

\bibitem{MykFact1}Ya.~V.~Mykytyuk, \textit{Factorization of Fredholm operators} Mat. Stud. \textbf{20(2)} (2003), 185--99 (in Ukrainian).

\bibitem{MykFact2}Ya.~V.~Mykytyuk, \textit{Factorization of Fredholm operators in operator algebras} Mat. Stud. \textbf{21(1)} (2004), 87--97 (in Ukrainian).

\bibitem{kreinvolterra}I.~Gokhberg and M.~Krein, \textit{Theory of Volterra operators in Hilbert space and its applications} Nauka, 1967 (in Russian).

\bibitem{young}R.~M.~Young, {\it An introduction to nonharmonic Fourier series} Revised 1st Edition, Academic Press, 2001.

\end{thebibliography}
\end{document}